\documentclass[a4]{amsart}

\usepackage[utf8]{inputenc}
\usepackage[english]{babel}
\usepackage{amssymb}
\usepackage{mathtools}
\usepackage[mathscr]{euscript}
\usepackage{stmaryrd}
\usepackage{enumitem}
\setenumerate[1]{label=(\alph*)}
\setenumerate[2]{label=(\roman*)}
\usepackage[all]{xy}
\usepackage{tikz}
\usetikzlibrary{matrix,arrows}
\usepackage{rotating}
\usepackage{extarrows}
\usepackage{accents}
\usepackage{leftidx}
\usepackage{upgreek}
\usepackage{bbm}
\usepackage{bbold}
\usepackage{xr-hyper}
\usepackage{hyperref}
\usepackage{macros}

\hypersetup{  
  pdftitle    = {Geometricity for derived categories of algebraic stacks},
  pdfauthor   = {Daniel Bergh, Valery A.~Lunts, Olaf M.~Schn{\"u}rer},
  colorlinks=true}

\numberwithin{equation}{section}
\title[Geometricity for derived categories of algebraic stacks]{Geometricity
for derived categories\\ of algebraic stacks}
\begin{document}
\dedicatory{To Joseph Bernstein on the occasion of his 70th birthday}
\begin{abstract}
We prove that the dg category of perfect
complexes on 
a smooth, proper Deligne--Mumford stack over a field of
characteristic zero
is geometric in the sense of Orlov, and in particular smooth and proper.
On the level of triangulated categories, this means that
the derived category of 
perfect complexes embeds as an admissible subcategory into
the bounded derived category of coherent sheaves on a smooth,
projective variety.
The same holds for a smooth, projective, tame Artin stack over
an arbitrary field.
\end{abstract}


\author{Daniel Bergh \and Valery A.~Lunts \and Olaf M.~Schn{\"u}rer}

\address{       
  Mathematisches Institut\\
  Universität Bonn\\
  Endenicher Allee 60\\
  53115 Bonn\\ 
  Germany
}
\email{dbergh@gmail.com}

\address{
  Department of Mathematics\\
  Indiana University\\
  Rawles Hall\\
  831 East 3rd Street\\
  Bloomington, IN 47405\\
  USA
}
\email{vlunts@indiana.edu}

\address{       
  Mathematisches Institut\\
  Universität Bonn\\
  Endenicher Allee 60\\
  53115 Bonn\\ 
  Germany
}
\email{olaf.schnuerer@math.uni-bonn.de}
\subjclass[2010]{Primary 14F05; Secondary 14A20, 16E45}
\keywords{Differential graded category, Derived category, Algebraic stack,
Root construction, Semiorthogonal decomposition}
\maketitle
\setcounter{secnumdepth}{1}
\setcounter{tocdepth}{1}
\tableofcontents

\section{Introduction}

The derived category of a variety
or, more generally, of an algebraic stack,
is traditionally studied in the context of triangulated categories.
Although triangulated categories are certainly powerful,
they do have some shortcomings.
Most notably, the category of triangulated
categories seems to have no tensor product, no concept of duals,
and categories of triangulated functors have no obvious
triangulated structure. 
A remedy to these problems is to work instead with differential
graded categories, also called dg~categories. 
We follow this approach and replace 
the derived category 
$\Dpf(X)$ of perfect complexes on 
a variety or an algebraic stack $X$
by a certain dg category $\Ddgpf(X)$ which enhances
$\Dpf(X)$ in 
the sense that its homotopy category is
equivalent to $\Dpf(X)$. 

The study of dg~categories is central in noncommutative geometry, 
and dg~categories are sometimes thought of as
categories of sheaves on a hypothetical noncommutative space.
Although a variety in general cannot be recovered from its
associated dg~category,
several of its important homological invariants can.
These include the algebraic K-theory spectrum
as well as various variants of cyclic homology.
See \cite{keller2006} and~\cite{tabuada2011} for surveys on
dg~categories and their invariants.

In noncommutative algebraic geometry,
saturated dg~categories play a similar role
as smooth and proper varieties in usual commutative
algebraic geometry.
For example, the dg~category $\Ddgpf(X)$ associated to a
variety $X$ is saturated if and only if $X$ is smooth and proper.
The saturated dg~categories have an intrinsic characterization
as the homotopy dualizable objects in the category of all
dg~categories with respect to a certain localization
\cite[§5]{ct2012}. 

It is natural to ask how far a dg category, thought of as a
noncommutative space, is from being commutative.  
Motivated by the dominant role of smooth, projective varieties, 
Orlov recently introduced the notion of a geometric
dg category \cite{orlov2014}.
By definition, every dg category of the form 
$\Ddgpf(X)$, for $X$ a smooth, projective variety, is geometric,
and so are all its ``admissible'' subcategories (see
Definition~\ref{d:geometric-NC-scheme} for a precise definition).
Every geometric dg category is saturated.
Orlov asks whether in fact all saturated dg categories 
are geometric.
To our knowledge, this question is wide open.
Our main theorems say that we stay in the realm of geometric dg
categories if we consider certain algebraic stacks.
{
\renewcommand{\thetheorem}{\ref{thm:main-zero}}
\begin{theorem}
Let $X$ be a smooth, proper Deligne--Mumford stack over a
field of characteristic zero.
Then the dg~category $\Ddgpf(X)$ is geometric, and in particular saturated.
\end{theorem}
}

This theorem can be seen as a noncommutative counterpart to
\cite[Corollary~4.7]{choudhury2012},
which states that the mixed motive of a
finite type, smooth Deligne--Mumford stack over a field of characteristic zero
is effective geometric.

In the preprint \cite{hlp2015}, the authors consider a
similar problem for certain stacks with positive dimensional
stabilizers and even for categories
of matrix factorizations on such stacks. 
They prove a geometricity result in a generalized sense
which involves 
infinite sums of dg categories of smooth Deligne--Mumford stacks
\cite[Theorem~2.7]{hlp2015}.  
As mentioned in Remark~2.9 of \emph{loc.\ cit.}\ our result
strengthens their Theorem~2.7 by replacing Deligne--Mumford
stacks with varieties.


We also give a version of our main theorem which is valid for
stacks over arbitrary fields.
Since we do not have resolution of singularities
over a field of positive characteristic,
we restrict the discussion to \emph{projective}
algebraic stacks (see Definition~\ref{def-projective}).
Indeed, even for a smooth, proper scheme $X$ over a field of
positive characteristic it is not clear whether $\Ddgpf(X)$
is geometric if $X$ is not projective.
Moreover, instead of Deligne--Mumford stacks we consider tame
algebraic 
stacks \cite{aov2008}.
Over a field of characteristic zero the
class of separated Deligne--Mumford stacks coincides with the class of
separated tame algebraic stacks, but in positive characteristic
tame stacks are usually better behaved.
For example, under mild finiteness assumptions
the tame algebraic stacks are scheme-like from a
noncommutative perspective
in the sense that their derived categories are
generated by a single compact object and the compact
objects coincide with the perfect
complexes~\cite[Theorem~3.1.1]{bv2003}, \cite[Theorem~A, Remark~4.6]{hr2014}.
{
\renewcommand{\thetheorem}{\ref{thm:main-tame}}
\begin{theorem}
Let $X$ be a tame, smooth, projective algebraic stack
over an arbitrary field.
Then the dg~category $\Ddgpf(X)$ is geometric, and in particular saturated.
\end{theorem}
}

We specify the dg enhancement $\Ddgpf(X)$ of $\Dpf(X)$ we work
with in Example~\ref{expl:algebraic-stack-enhancement}.
A priori, there are other possible choices (cf.\ Remark~\ref{rem:drinfeld-quotient}), but it turns out that
they are all equivalent.
This follows from a recent result by Canonaco and~Stellari
\cite{canonaco-stellari-uniqueness-of-dg-enhancements}
which implies that the derived categories $\Dpf(X)$
for the stacks considered in the theorems above have
unique dg enhancements
(see~Remark~\ref{rem:Dpf-unique-enhancemenet}). 

Theorem~\ref{thm:main-tame} has the following equivalent
reformulation in terms of varieties with group actions.

{
\renewcommand{\thetheorem}{\ref{thm:main-tame}b}
\begin{theorem}
Let $U$ be a smooth, quasi-projective variety over a field $\kk$,
and let $G$ be a linear algebraic group over $\kk$ acting properly
on $U$.
Assume that the action admits a geometric quotient
$U \to U/G$ 
(in the sense of \cite[Definition~0.6]{mfk1994})
such that $U/G$ 
is projective over $\kk$.
Also assume that all stabilizers of the action are linearly reductive.
Then the dg~category enhancing the bounded derived
category $\D^\mathrm{b}(\Coh^G(U))$ of $G$-equivariant
coherent sheaves on $U$ is geometric, and in particular saturated.
\end{theorem}
\begin{proof}
See Example~\ref{ex:git} and Remark~\ref{rem:dqc} together with
the fact that the category of coherent sheaves on the stack $[U/G]$ is
equivalent to the category of $G$-equivariant coherent sheaves on $U$.
\end{proof}
}

\noindent
Note that the requirement that the action be proper implies
that the stabilizers are finite.
In particular, the condition that the stabilizers be linearly
reductive is superfluous if our base field $\kk$ has
characteristic zero.
Also note that if $G$ is finite and $U$ is projective over $\kk$,
then the action is automatically proper and the existence of
a projective geometric quotient $U/G$ is guaranteed.

\subsection{Outline}
The proof of the main results primarily builds
on two results --
the destackification theorem by Bergh and Rydh
\cite{bergh2014}, \cite{br2015} and the gluing
theorem for geometric dg~categories by Orlov~\cite{orlov2014}.

The destackification theorem allows us to compare a smooth,
tame algebraic stack to a smooth algebraic space via a sequence of
birational modifications called \define{stacky blowups}.
We review this theorem as Theorem~\ref{thm:destack}
in Section~\ref{sec:geometry}.
In this section,
we also give the main geometric arguments of
the proofs of the main theorems.

Stacky blowups play a similar role in the study of the birational geometry
of tame stacks as do usual blowups for schemes.
They come in two incarnations: usual blowups and so-called
\define{root stacks}.
Root stacks are purely stacky operations which have no counterpart
in the world of schemes.
We review the notion of a root stack
in Section~\ref{sec:root-constructions}.

A stacky blowup modifies the derived category of a stack in
a predictable way.
More specifically, it induces a semiorthogonal decomposition on
the derived category on the blowup.
For usual blowups this is due to Orlov~(\cite{orlov1992},
\cite[Proposition~11.18]{huybrechts2006})
and for root stacks this is due to Ishii--Ueda~\cite{iu2011}.
In Section~\ref{sec:semiorthogonal},
we reprove the theorem by Ishii--Ueda as Theorem~\ref{thm:iishi-ueda},
but in a more general setting.
We also give a combinatorial description of the semiorthogonal
decomposition on an iterated root stack as
Theorem~\ref{thm:semi-orthogonal-iterated}.

In Section~\ref{sec:diff-grad-enhancements}, we provide the dg
categorical ingredients for the proofs of our main theorems.
We first introduce the dg enhancements $\Ddgpf(X)$ and lift certain
derived functors to these enhancements. Then we discuss 
geometric dg~categories and state Orlov's gluing theorem as
Theorem~\ref{thm:orlov-with-converse}. 

Some general facts concerning algebraic stacks, derived
categories, and semiorthogonal decompositions are assembled in
section~\ref{sec:preliminaries}.  
Appendix \ref{sec:bound-deriv-coherent} contains some results on
bounded derived categories of coherent sheaves on noetherian
algebraic stacks.

\subsection{Acknowledgments}
\label{sec:acknowledgment}
We thank David Rydh for detailed comments.
Daniel Bergh was partially supported by Max Planck Institute for Mathematics,
Bonn,
and by the DFG through SFB/TR 45.
Valery Lunts was partially supported by the NSA grant 141008.
Olaf Schn\"urer was partially supported by the DFG through a
postdoctoral fellowship and through SPP 1388 and SFB/TR 45.


\section{Preliminaries}
\label{sec:preliminaries}
In this section,
we fix our notation and our conventions for algebraic stacks
as well as their derived categories of sheaves.
We also review the notions of tame and projective stacks.
Finally, we review the notion of semiorthogonal decompositions
for triangulated categories.

\subsection{Conventions for algebraic stacks}
We will use the definition of algebraic space and algebraic stack
from the stacks project~\cite[\sptag{025Y}, \sptag{026O}]{stacks-project}.
In particular, we will always state all separatedness assumptions
explicitly.
The main results of this article concern tame algebraic
stacks which are separated and of finite type over a field $\kk$.
If $\kk$ is a field of characteristic zero,
the class of these stacks consists precisely of the
separated Deligne--Mumford stacks of finite
type over $\kk$. 
Thus the reader unwilling to work in full generality
could safely assume that all stacks are of the aforementioned kind.

Although algebraic stacks form a 2-category
we will follow the common practice to suppress their
2-categorical nature to simplify the exposition if
no misunderstanding is likely.
In particular, we will often say morphism instead of 1-morphism,
isomorphism instead of equivalence,
commutative instead of 2-commutative and cartesian
instead of 2-cartesian.

\subsection{Tame algebraic stacks}
An algebraic stack $X$ has \define{finite inertia} if
the canonical morphism $I_X \to X$ from its inertia stack is finite.
If $X$ has finite inertia,
then there is a canonical morphism $\pi\colon X \to
X_\coarse$ to the \define{coarse (moduli) space},
which is an algebraic space \cite{km1997, rydh2013}.
If $X$ is locally of finite type over some base algebraic space $S$,
then the morphism $\pi$ is proper.

\begin{definition}
An algebraic stack $X$ is called \define{tame} provided that the stabilizer
at each of its geometric points is finite and linearly reductive.
\end{definition}
If $X$ has finite inertia, tameness implies that the pushforward
$\pi_\ast\colon \Qcoh(X) \to \Qcoh(X_\coarse)$ is exact.
The converse implication holds if $X$ has finite inertia and is
quasi-separated~\cite[Corollary~A.3]{hall2014}.

\begin{remark}
Tame algebraic stacks are defined in \cite{aov2008} in a slightly less general
context.
We use the more general definition given in \cite{hall2014}.
\end{remark}

\begin{example}
Let $X$ be a separated algebraic stack of finite type over a field $\kk$.
Then $X$ has finite inertia provided that the stabilizer at each geometric point
is finite. 
If $\kk$ has characteristic zero, then $X$ is tame if and only if it
is a Deligne--Mumford stack.
If $\kk$ has characteristic $p > 0$, the
stack $X$ is tame if and only if the stabilizer group at each
geometric point is of the form
$\Delta \rtimes H$,
where $\Delta$ is a finite diagonalizable $p$-group and $H$
is a constant finite group of order prime to $p$.
\end{example}

\subsection{Projective algebraic stacks}
A definition of \define{projective} Deligne--Mumford
stack is suggested in \cite{kresch2009}.
Since we also work with some tame stacks which are not Deligne--Mumford stacks,
we will need to extend this definition slightly.
First we clarify what we mean by a global quotient stack.

An algebraic stack $X$ is a \define{global quotient stack}
if there exists a $\GL_n$-torsor $T \to X$,
for some non-negative integer $n$, 
such that $T$ is an algebraic space.

\begin{example}
\label{ex:global-quotient}
Let $\kk$ be a field and $G \subset \GL_n$ a linear algebraic group
over $\kk$.
Assume that $G$ acts on an algebraic space $U$ over $\kk$.
Let $T$ be the quotient of $\GL_n\times U$ by $G$,
with $G$ acting on the right on the factor $\GL_n$ via the inclusion $G \subset
\GL_n$ and on the left on $U$.
Then $T$ is an algebraic space and the obvious projection $T \to [U/G]$
is a $\GL_n$-torsor.
In particular, the stack quotient $[U/G]$ is a global quotient stack.
\end{example}

If a global quotient stack $X$ is separated over some
base algebraic space $S$ then the relative diagonal
$\Delta_{X/S}\colon X \to X\times_S X$
is affine and proper and hence finite.
In particular, such a stack has finite
inertia and therefore admits a coarse 
space $X \to X_\coarse$.

\begin{definition}
\label{def-projective}
Let $X$ be an algebraic stack over a field $\kk$.
We say that $X$ is \define{quasi-projective over $\kk$}
if it is a global quotient stack which is separated and of
finite type over $\kk$,
and its coarse space $X_\coarse$ is a quasi-projective scheme over $\kk$.
If in addition $X$ or, equivalently, $X_\coarse$
is proper over $\kk$,
we say that $X$ is \define{projective over~$\kk$}. 
\end{definition}

\begin{example}
\label{ex:git}
GIT-quotients by proper actions
(in the sense of \cite{mfk1994})
give rise to quasi-projective algebraic stacks.
Let $U$ be a quasi-projective variety over a field $\kk$
and let $G$ be a reductive linear algebraic group acting
properly on $U$.
Assume that $U$ admits a $G$-linearized line bundle
$\mathcal{L}$ such that $U$ is everywhere stable with respect
to $\mathcal{L}$ in the sense of \cite[Definition~1.7]{mfk1994}.
Then the GIT-quotient $U \to U/G$ is geometric and $U/G$
is quasi-projective.
Since geometric quotients by proper actions are universal among
algebraic spaces \cite[Corollary~2.15]{kollar1997},
the canonical morphism $[U/G] \to U/G$ identifies $U/G$
with the coarse space of the stack quotient $[U/G]$.
Since the action of $G$ on $U$ is proper,
the stack quotient $[U/G]$ is separated.
Hence $[U/G]$ is a quasi-projective stack in the sense of
Definition~\ref{def-projective}.

Conversely, every quasi-projective algebraic stack can be
obtained in this way.
Indeed, assume that $X$ is quasi-projective.
Since $X$ is a separated global quotient stack,
it is of the form $[U/\GL_n]$
where $\GL_n$ acts properly on an algebraic space $U$.
Denote the geometric quotient by
$q\colon U \to U/\GL_n \cong [U/\GL_n]_\coarse$.
Let $\mathcal{M}$ be an ample line bundle on $U/\GL_n$.
This pulls back to a $\GL_n$-linearized line bundle
$\mathcal{L} = q^\ast\mathcal{M}$ on $U$.
Since $q$ is affine (cf.~\cite[Remark~4.3]{kresch2009} or
\cite[Theorem~3.12]{kollar1997}),
the bundle $\mathcal{L}$ is ample and $U$ is quasi-projective.
Moreover, the space $U$ is everywhere stable with respect to
$\mathcal{L}$, since sections of $\mathcal{M}$ pull back to
invariant sections of $\mathcal{L}$.
\end{example}

\begin{example}
Let $U$ be a quasi-projective variety over a field $\kk$,
and let $G$ be a finite group scheme over $\kk$ acting on $U$.
Then the quotient $U/G$ is quasi-projective by
\cite[Exposé~V, Proposition~1.8]{SGA-1}
combined with graded prime avoidance (cf.\ \cite[\sptag{09NV}]{stacks-project}).
Since $U/G$ coincides with the coarse space of $[U/G]$,
it follows that $[U/G]$ is a quasi-projective stack in
the sense of Definition~\ref{def-projective}.
Moreover, since the natural morphism $[U/G] \to U/G$ is proper,
the stack $[U/G]$ is projective if and only if $U$ is projective.
\end{example}

As expected, we have the following permanence property for projective algebraic
stacks with respect to morphisms which are projective in the
sense of \cite[Definition~5.5.2]{EGAII}.
\begin{lemma}
\label{lem:projective-stack}
Let $f\colon X \to Y$ be a projective morphism of algebraic
stacks with $Y$ being quasi-compact and quasi-separated.
Assume that both $X$ and $Y$ have finite inertia.
Then the induced morphism $f_\coarse\colon X_\coarse \to Y_\coarse$
between the coarse spaces is projective.
In particular, if $Y$ is (quasi-)projective over a field $\kk$ in the
sense of Definition~\ref{def-projective},
then the same holds for~$X$.
\end{lemma}
\begin{proof}
Note that since $Y$ is assumed to be quasi-compact and
quasi-separated, projectivity of $f$ is equivalent to $f$ being
proper and admitting an $f$-ample invertible sheaf (see \cite[Proposition~8.6]{rydh2015}).
Hence the statement of the lemma follows from \cite[Proposition
2]{mo206117}. 
Less general versions of the lemma can be found in \cite[Proof of
Theorem~1]{kv2004} and \cite[Proposition~6.1]{olsson2012}.
\end{proof}

\subsection{Derived categories and derived functors}
There are several equivalent ways to define quasi-coherent modules on
algebraic stacks.
We will follow \cite{lmb2000} and view quasi-coherent modules as
sheaves on the lisse--étale site.
The results on derived categories depend on the techniques
of cohomological descent as described in \cite{olsson2007} and~\cite{lo2008}.
A concise summary of these results is given in \cite[Section~1]{hr2014}.
We give a brief overview here.
Let $(X, \mathcal{O})$ be a ringed topos.
We use the notation
$
\Mod(X, \mathcal{O}),
$
for the abelian category of $\mathcal{O}$-modules in $X$ and
$
\D(X, \mathcal{O})
$
for its derived category.

Let $X$ be an algebraic stack.
We denote the topos of sheaves on the lisse--étale site by $X_\liset$.
If $X$ is a Deligne--Mumford stack, we denote the topos of sheaves
on the small étale site by $X_\et$.
In these situations, we use the short hand notation
$$
\Mod(X_\tau) := \Mod(X_\tau, \mathcal{O}_X),
\qquad
\D(X_\tau):=  \D(X_\tau, \mathcal{O}_X),
$$
where $\tau$ is either $\liset$ or $\et$.
By default, we will use the lisse-étale site when considering sheaves on
algebraic stacks and simply write $\Mod(X)$ instead of $\Mod(X_\liset)$ and
$\D(X)$ instead of $\D(X_\liset)$.

Recall that an $\mathcal{O}_X$-module is \define{quasi-coherent} if
it is locally presentable \cite[\sptag{03DL}]{stacks-project}.
We let $\Qcoh(X)$ denote the full subcategory of $\Mod(X)$ of quasi-coherent
modules and $\Dqc(X)$ the full subcategory of $\D(X)$ consisting of complexes
with quasi-coherent cohomology.
Since $\Qcoh(X)$ is a weak Serre subcategory of $\Mod(X)$,
the category $\Dqc(X)$ is a thick triangulated subcategory of $\D(X)$.

Also recall that a complex in $\Mod(X)$ is called
\define{perfect} if it is locally quasi-isomorphic to a bounded
complex of direct summands of finite free modules
\cite[\sptag{08G4}]{stacks-project}.
We denote by $\Dpf(X)$ the subcategory of $\D(X)$ consisting of perfect
complexes.
The category $\Dpf(X)$ is a thick triangulated subcategory of $\Dqc(X)$. 

\begin{remark}
\label{rem:DM-compare-liset-small-etale}
The reader willing to restrict the discussion to Deligne--Mumford stacks
could instead use $X_\et$ as the default topos when considering sheaves
on such a stack $X$.
In this situation, the correspondingly defined categories
$\Dqc(X_\et)$ and~$\Dpf(X_\et)$ are equivalent to
$\Dqc(X)$ and~$\Dpf(X)$ respectively.

Explicitly, the equivalences are constructed as follows.
Let $X$ be a Deligne--Mumford stack.
The inclusion of its small 
étale site into its lisse-étale site induces a morphism 
\begin{equation}
\label{eq:7}
\epsilon = (\epsilon^\ast, \epsilon_\ast) \colon
(X_\liset, \mathcal{O}_X) \ra (X_\et, \mathcal{O}_X)
\end{equation}
of ringed topoi,
where $\epsilon_\ast$ is the restriction functor.
Both functors $\epsilon^\ast$ and $\epsilon_\ast$ are exact,
and the equivalences are obtained by restriction of the induced
adjoint pair
$\epsilon^* \colon \D(X_\et) \ra \D(X_\liset)$
and 
$\epsilon_* \colon \D(X_\liset) \ra \D(X_\et)$. 
\end{remark}


Let $f\colon X \to Y$ be a morphism of algebraic stacks.
Assume, for simplicity, that $f$ is \define{concentrated}.
This means that $f$ is quasi-compact, quasi-separated and
has a boundedness condition on its cohomological dimension
$Y$~\cite[Definition~2.4]{hr2014}.
For our needs, it suffices to know that a quasi-compact
and quasi-separated morphism of algebraic stacks is concentrated
provided that its fibers are tame (cf.~\cite[Theorem~2.1]{hr2015}).
We get induced adjoint pairs of functors  
$$
f^\ast\colon \Qcoh(Y) \to \Qcoh(X), \qquad
f_\ast\colon \Qcoh(X) \to \Qcoh(Y)
$$
and
$$
\L f^\ast\colon \Dqc(Y) \to \Dqc(X), \qquad
\R f_\ast\colon \Dqc(X) \to \Dqc(Y).
$$
Here $f^\ast$, $f_\ast$ and $\R f_\ast$ are simply the restrictions
of the corresponding functors on $\Mod(X)$,
$\Mod(Y)$ and $\D(X)$ (\cite[Theorem~2.6.(2)]{hr2014}).

\begin{remark}
\label{rem:left-derived-liset}
It requires some work to see that the functor
$\L f^\ast \colon \Dqc(Y) \to \Dqc(X)$ actually exists.
This is due to the fact that the naturally defined
adjoint pair $(f^{-1}, f_\ast)$ does not induce
a morphism 
$(X_\liset, \mathcal{O}_X) \to (Y_\liset, \mathcal{O}_Y)$
of ringed topoi,
owing to the fact that $f^{-1}$ in general is not exact.
Hence, we do not get a functor $\L f^\ast \colon \D(Y) \to \D(X)$
from the general theory.
\end{remark}

\begin{remark}
\label{rem:left-derived-et}
If $f\colon X \to Y$ is a morphism of Deligne--Mumford stacks,
then the pair $(f^{-1}, f_\ast)$ does induce a morphism
$(X_\et, \mathcal{O}_X) \to (Y_\et, \mathcal{O}_Y)$
of ringed topoi.
Hence we do get a functor
$\L f^\ast \colon \D(Y_\et) \to \D(X_\et)$.
Furthermore, its restriction to $\Dqc(Y_\et)$
is compatible with  
$\L f^\ast \colon \Dqc(Y) \to \Dqc(X)$
via the equivalences described in Remark~\ref{rem:DM-compare-liset-small-etale}.
\end{remark}

\begin{remark}
\label{rem:dqc}
Let $X$ be an algebraic stack.
Then the category $\Qcoh(X)$ is a Grothendieck abelian category.
In particular, the category of complexes of quasi-coherent modules
has enough h-injectives and the derived category $\D(\Qcoh(X))$
has small hom-sets.

There is an obvious triangulated functor
\begin{equation}
  \label{eq:10}
  \D(\Qcoh(X)) \to \Dqc(X) 
\end{equation}
induced by the
inclusion $\Qcoh(X) \subset \Mod(X)$.
Assume that $X$ is quasi-compact, separated and has finite stabilizers.
In particular, this implies that $X$ has finite, and hence affine, diagonal.
Then the functor \eqref{eq:10} is an equivalence of
categories.
This follows from \cite[Theorem~1.2]{hnr2014}
and~\cite[Theorem~A]{hr2014}.

Assume that $X$, in addition, is regular.
In particular, this includes the stacks considered in the main theorems of this
article.
Then the obvious functor induces an equivalence
$\D^{\mathrm{b}}(\Coh(X)) \cong \Dpf(X)$ by
Remark~\ref{rem:D-Coh-and-DbCoh=Dpf}.
In particular, a complex of $\mathcal{O}_X$-modules is perfect if
and only 
if it is
isomorphic in $\Dqc(X)$ to a bounded complex of coherent modules.
\end{remark}

\begin{remark}
\label{rem:dqc2}
Let $X \to Y$ be a concentrated morphism of
quasi-compact stacks which are separated and have finite stabilizers.
Then the derived pushforward $\R f_\ast\colon\D(\Qcoh(X)) \to
\D(\Qcoh(Y))$ corresponds to $\R f_\ast\colon\Dqc(X) \to \Dqc(Y)$ under the
equivalences mentioned in
Remark~\ref{rem:dqc}, see \cite[Corollary~2.2]{hnr2014}.
\end{remark}

\subsection{Semiorthogonal decompositions}
We recall the definition of admissible subcategories and semiorthogonal
decompositions of triangulated categories
(cf.~\cite{bondal-kapranov-representable-functors},
\cite[Appendix~A]{valery-olaf-matfak-semi-orth-decomp}).

\begin{definition}
Let $\mathcal{T}$ be a triangulated category.
A \define{right} (resp.\ \define{left}) \define{admissible
subcategory of} $\mathcal{T}$ 
is a strict full triangulated subcategory $\mathcal{T}'$ of $\mathcal{T}$
such that the inclusion functor 
$\mathcal{T}' \to \mathcal{T}$ admits a right (resp.\ left) adjoint.
An \define{admissible} subcategory is a subcategory that is both
left and right admissible.
\end{definition}
\begin{definition}
A sequence $(\mathcal{T}_1, \ldots, \mathcal{T}_r)$ of
subcategories of
$\mathcal{T}$ is called \define{semiorthogonal} provided that
$\Hom_\mathcal{T}(T_i, T_j) = 0$ for all objects
$T_i \in \mathcal{T}_i$ and $T_j \in \mathcal{T}_j$ whenever $i > j$.
If, in addition, all $\mathcal{T}_i$ are strict full triangulated
subcategories and the category $\mathcal{T}$ coincides with its smallest
strict full triangulated subcategory containing all the 
$\mathcal{T}_i$, 
then we say that the sequence forms a
\define{semiorthogonal decomposition} of $\mathcal{T}$
and write
$$
\mathcal{T} = \langle \mathcal{T}_1, \ldots, \mathcal{T}_r \rangle.
$$
We say that a sequence $\Phi_1, \ldots, \Phi_r$ of 
triangulated 
functors with codomain $\mathcal{T}$
forms a \define{semiorthogonal decomposition} of $\mathcal{T}$ 
and write $\mathcal{T} = \langle\Phi_1, \ldots, \Phi_r\rangle$
if all $\Phi_i$ are 
full and faithful
and the essential images of the functors $\Phi_1,
\dots, \Phi_r$ form a
semiorthogonal decomposition of $\mathcal{T}$.
\end{definition}


\section{Root constructions}
\label{sec:root-constructions}
The root construction can be seen as a way of adjoining
roots of one or several divisors
on a scheme or an algebraic stack.
It has been described in several sources,
e.g.~\cite[§2]{cadman2007}, \cite[§2.1]{bc2010},
\cite[§1.3]{fmn2010} and~\cite[Appendix~B]{agv2008}.
We recall its definition along with some of its basic properties.
Since most of these basic properties are already described in
the sources mentioned above or trivial generalizations, 
we will omit most of the proofs.

It is straightforward to define the root stack in terms of its
generalized points (cf.~\cite[Remark~2.2.2]{cadman2007}),
but the most economical definition seems to use the
\define{universal root construction}.
Let $r$ be a positive integer and consider the commutative diagram
\begin{equation}
\label{eq:universal-root-diagram}
\xymatrix{
\B\GGm \ar[r]^-{\iota} \ar[d]_-{\rho}
& 
[\AA^1/\GGm] \ar[d]^-{\pi}
\\
{\B\GGm} \ar[r]
&
{[\AA^1/\GGm].}
}
\end{equation}
Here $[\AA^1/\GGm]$ denotes the stack quotient of
$\AA^1 = \Spec \ZZ[x]$ by $\GGm$ acting by multiplication.  
The maps $\rho$ and $\pi$ are induced by the maps $\GGm \to \GGm$
and $\AA^1 \to \AA^1$ taking the coordinate $x$ to its $r$-th
power $x^r$.  
Note that the diagram above is not cartesian if $r>1$.

Recall that the stack $[\AA^1/\GGm]$ parametrizes pairs consisting of
a line bundle together with a global section.

\begin{definition}
\label{def:root-construction}
Let $X$ be an algebraic stack and $E$ an effective Cartier divisor
\cite[\sptag{01WR}]{stacks-project} on $X$.
Consider the morphism $f_E\colon X \to [\AA^1/\GGm]$
corresponding to the line bundle $\mathcal{O}_X(E)$
together with the canonical global section $\mathcal{O}_X \to \mathcal{O}_X(E)$.
Given a positive integer $r$,
we construct the \define{root diagram}
\begin{equation}
\label{eq:root-diagram}
\xymatrix{
{r^{-1}E} \ar[r]^-{\iota} \ar[d]_-{\rho}
& 
{X_{r^{-1}E}} \ar[d]^-{\pi}
\\
{E} \ar[r]^{\kappa}
&
{X}
}
\end{equation}
as the base change of the universal root
diagram~(\ref{eq:universal-root-diagram})
along the morphism $f_E$.
The stack $X_{r^{-1}E}$ is called the \define{$r$-th root stack}
of $X$ with respect to $E$. 
The notation for the divisor $r^{-1}E$ is motivated by the fact
that $r \cdot (r^{-1}E)=\pi^*E$.
We refer to this construction as the \define{root construction}
with respect to the datum $(X, E, r)$.
\end{definition}

The next example gives a local description of a root stack.
\begin{example}
\label{ex:root-local-description}
Assume that $X = \Spec R$ is affine and the effective Cartier
divisor $E \hookrightarrow X$ corresponds to a
ring homomorphism $R \to R/(f)$ where $f\in R$ is a regular element.
Then the $r$-th root construction yields
\begin{equation}
\label{eq:local-coordinates}
r^{-1}E = \B\rmu_r\times E, \qquad
X_{r^{-1}E} = [\Spec R'/\rmu_r], \qquad R' = R[t]/(t^r-f)
\end{equation}
where $\rmu_r$ denotes the group scheme of $r$-th roots of unity.
The $\rmu_r$-action on $\Spec R'$ corresponds to the
$\mathbb{Z}/r\mathbb{Z}$-grading on $R'$ with $R$ in degree zero
and $t$ homogeneous of degree~1.
The closed immersion $r^{-1}E \hookrightarrow X_{r^{-1}E}$ corresponds
to the ideal generated by $t$. 
\end{example}

\begin{proposition}
\label{prop:basic-root}
The root construction described in Definition~\ref{def:root-construction}
has the following basic properties:
\begin{enumerate}
\item
  \label{enum:cartier}
The morphism $\iota$ in diagram~\eqref{eq:root-diagram} is a closed
immersion realizing $r^{-1}E$ as an effective Cartier divisor on $X_{r^{-1}E}$.
\item
  \label{enum:pi-proper-fp-birat}
The morphism $\pi$ is a universal homeomorphism which is proper,
faithfully flat and birational with exceptional locus contained in $r^{-1}E$.
\item
  \label{enum:rho}
The morphism $\rho$ turns $r^{-1}E$ into a $\rmu_r$-gerbe over $E$ with
trivial Brauer class.
In particular, the morphism $\rho$ is smooth.
\item
  \label{enum:coarse}
  If $X$ is
  an algebraic space, then $\pi$ identifies $X$ with the coarse
  space 
  of $X_{r^{-1}E}$.
  More generally, if $X$ is an algebraic stack then the morphism $\pi$
  is a \emph{relative} coarse space.
  In particular, if
  $X$ is an algebraic stack having a coarse space $X \ra
  X_\coarse$, then 
  the composition $X_{r^{-1}E} \ra X \ra X_\coarse$ is a coarse
  space for $X_{r^{-1}E}$.  
\item
  \label{enum:tame}
The pushforward $\pi_\ast \colon \Qcoh(X_{r^{-1}E}) \ra \Qcoh(X)$ is exact,
and $X_{r^{-1}E}$ is tame provided that the same holds for $X$.
\item
If $X$ is a Deligne--Mumford stack and $r$ is invertible in $\mathcal{O}_X$,
then $X_{r^{-1}E}$ is a Deligne--Mumford stack.
\end{enumerate}
\end{proposition}

The notion of projectivity is well-behaved under taking roots of
effective Cartier divisors.
\begin{lemma}
\label{lem:projective-root}
Let $X$ be a quasi-projective algebraic stack over a field $\kk$,
and let $E$ be an effective Cartier divisor on $X$.
Then the root stack $X_{r^{-1}E}$ is quasi-projective over $\kk$ for any positive
integer $r$. 
In particular, if $X$ is projective over $\kk$, then so is $X_{r^{-1}E}$.
\end{lemma}
\begin{proof}
Assume that $X$ is quasi-projective over $\kk$.
Since $X_{r^{-1}E} \ra X$ is proper
(Proposition~\ref{prop:basic-root}.\ref{enum:pi-proper-fp-birat}), 
the root stack
$X_{r^{-1}E}$ is separated and of finite type over $\kk$. 
Since $X$ and $X_{r^{-1}E}$ have isomorphic 
coarse spaces
(Proposition~\ref{prop:basic-root}.\ref{enum:coarse}),
it is enough to verify that $X_{r^{-1}E}$ is a global quotient stack.
But this follows from the general fact, proved below, that the fibre product of
two global quotient stacks over any algebraic stack is a global
quotient stack.
This we apply to the morphism $f_E \colon X \to [\AA^1/\GGm]$
corresponding to $E$ (cf.\ Definition~\ref{def:root-construction}) and the
morphisms $\pi \colon [\AA^1/\GGm] \to [\AA^1/\GGm]$ from the universal root
diagram~\eqref{eq:universal-root-diagram}.

Now we prove the general fact about fiber products of global quotient stacks.
Let $Y \to S$ and $Z \to S$ be morphisms
of algebraic stacks and assume that $Y$ and $Z$ are global quotients.
Then there exist a $\GL_n$-torsor $U \to Y$ and a $\GL_m$-torsor
$V \to Z$ 
such that $U$ and $V$ are algebraic spaces.
Then $U\times_S V$ is an algebraic space
and the canonical morphism $U \times_S V \to Y\times_S Z$ is a
$\GL_n\times \GL_m$-torsor.
Extending this torsor along the obvious embedding
$\GL_n \times \GL_m \to \GL_{n+m}$
yields a $\GL_{n+m}$-torsor over $Y\times_S Z$ which is an
algebraic space. This shows that $Y\times_S Z$ is a global quotient.
\end{proof}

There is also the more general concept of a root stack in a simple normal
crossing (snc) divisor.
In the next section, we will obtain a semiorthogonal decomposition for
such root stacks.
Since this decomposition also exists in the case that the ambient
algebaic stack is not smooth,
will work with a (non-standard) generalized notion of snc divisor. 
This generality will not be needed in the applications we
have in mind, but it comes at no extra cost and 
reveals the true relative nature of the root construction
and the induced semiorthogonal decomposition.

\begin{definition}
\label{def:generalized-snc}
Let $X$ be an algebraic stack.
A \define{generalized snc divisor} on $X$ is a finite family
$E = (E_i)_{i \in I}$ of effective Cartier divisors on $X$
such that for each subset 
$J \subseteq I$ and each element
$i \in J$
the inclusion
$$
\cap_{j \in J}E_j \to \cap_{j \in J\setminus\{i\}}E_j  
$$
is an effective Cartier divisor.
We call the divisors $E_i$ the \define{components} of $E$.
If $X$ is smooth over a field $\kk$,
we call a generalized snc divisor an \define{snc divisor}
if all intersections $\cap_{i \in J}E_j$ are smooth over $\kk$.
\end{definition}

Note that the definition asserts that the non-empty components of
a generalized snc divisor are distinct.

\begin{remark}
Note that we do not require the components to be irreducible,
reduced or non-empty.
This will somewhat simplify the exposition in the proof of
Theorem~\ref{thm:semi-orthogonal-iterated}.
In the smooth case, our definition of an snc divisor coincides with the
standard one,
possibly with the subtle difference that we make the splitting of $E$
into components part of the structure.
\end{remark}

When dealing with the combinatorics of iterated root stacks,
it is convenient to use multi-index notation.
Given a finite set $I$ a \define{multi-index}
(with respect to $I$)
is an element $a = (a_i)$ of $\mathbb{Z}^I$.
Multiplication of multi-indexes is
defined coordinatewise.
We will also consider the partial ordering
on $\mathbb{Z}^I$ defined by $a \leq b$ if and only if
$a_i \leq b_i$ for all $i \in I$. We write $a < b$ if and only if
$a_i < b_i$ for all $i \in I$.
If $a$ is a multi-index and $E$ is a (generalized) snc divisor
indexed by $I$, then $aE$ denotes the Cartier divisor
given by $\sum_{i \in I} a_iE_i$.

\begin{definition}[Iterated root construction]
\label{def:iterated-root}
Let $X$ be an algebraic stack and $E = (E_i)_{i \in I}$ a
generalized snc divisor on $X$.
Fix a multi-index $r > 0$ in $\ZZ^I$.
The \define{$r$-th root stack}, denoted by $X_{r^{-1}E}$,
of $X$ with respect to $E$ and $r$ is defined as the fiber product
of the root stacks $X_{r_i^{-1}E_i}$ over $X$ for $i \in I$.
The \define{transform} of $E$ is defined as the family
$
r^{-1}E = (\widetilde{E}_i)_{i \in I},
$
where $\widetilde{E}_i$ denotes the pull-back of $r_i^{-1}E_i$
along the projection $X_{r^{-1}E} \to X_{r_i^{-1}E_i}$.
\end{definition}

\begin{remark}
If $X$ is smooth and $E$ is an snc divisor,
then the root stack $X_{r^{-1}E}$ only depends on the divisor $rE$.
The corresponding statement for generalized snc divisors is not true.
For instance, consider the coordinate axes $V(x)$ and $V(y)$ in the
affine plane $\AA^2 = \Spec \ZZ[x, y]$.
Then the $(2, 2)$-th root stack of $\AA^2$ in $(V(x), V(y))$ is smooth
whereas the $2$-nd root stack of $\AA^2$ in $(V(xy))$ is not
(cf.~\cite[§2.1]{bc2010}).
\end{remark}

The next proposition is well-known and stated, in a slightly different form,
in \cite[§2.1]{bc2010}.
We include a proof since none is given in \emph{loc.\ cit.}
\begin{proposition}
\label{prop:basic-iterated-root}
Let $X$ be an algebraic stack and $E$ a generalized snc divisor
indexed by $I$.
Given a multi-index $r > 0$ in $\DZ^I$, the $r$-th root construction of $X$ in $E$ has the
following properties:
\begin{enumerate}
\item \label{it:iterated-root}
Given $s > 0$ in $\DZ^I$,
the $s$-th root stack $(X_{r^{-1}E})_{s^{-1}(r^{-1}E)}$ of $X_{r^{-1}E}$ in the
transform $r^{-1}E$ is canonically isomorphic to the $(rs)$-th
root stack $X_{(sr)^{-1}E}$ in $E$,
and this isomorphism identifies $s^{-1}(r^{-1})E$ with $(sr)^{-1}E$.
\item \label{it:gen-snc}
The transform $r^{-1}E$ is a generalized snc divisor on $X_{r^{-1}E}$.
\end{enumerate}
If furthermore $X$ is smooth over a field $\kk$ and $E$ is an snc divisor,
then we have the following:
\begin{enumerate}[resume]
\item \label{it:smooth}
The stack $X_{r^{-1}E}$ is smooth and $r^{-1}E$ is an snc divisor.
\end{enumerate}
\end{proposition}
\begin{proof}
Assume that $r = (r_1, \ldots, r_n)$.
By identifying $[\AA^n/\GGm^n]$ with $[\AA^1/\GGm]^n$,
we get a canonical morphism
$$
\pi_r = \pi_{r_1}\times \cdots \times \pi_{r_n}\colon [\AA^n/\GGm^n] \to
[\AA^n/\GGm^n]
$$
where each $\pi_{r_i}$ corresponds to the morphism $\pi$ in
the universal $r_i$-th root diagram~\refeq{eq:universal-root-diagram}.
The generalized snc divisor $E = (E_1, \ldots, E_n)$ gives rise
to a morphism
$$
f_E\colon X \xrightarrow{\Delta} X^n \to [\AA^n/\GGm^n],
$$
where the second morphism is the product
$f_{E_1}\times \cdots \times f_{E_n}$ with each $f_{E_i}$ as
in Definition~\ref{def:root-construction}.
It is now easy to see that the structure map $X_{r^{-1}E} \to X$
of the root stack is canonically isomorphic to the pull-back of
$\pi_r$ along $f_E$ (cf.\ \cite[§1.3.b]{fmn2010}).
Using this description, statement \ref{it:iterated-root} follows from the
diagram
$$
\xymatrix{
(X_{r^{-1}E})_{s^{-1}(r^{-1}E))} \ar[r]\ar[d] & [\AA^n/\GGm^n] \ar[d]_{\pi_s} 
\\
X_{r^{-1}E} \ar[r]_{f_{r^{-1}E}}\ar[d] & [\AA^n/\GGm^n] \ar[d]_{\pi_r} \\
X \ar[r]_{f_{E}} & [\AA^n/\GGm^n]  \\
}
$$
with cartesian squares and the fact that $\pi_r\circ \pi_s = \pi_{rs}$.

Next we prove statement \ref{it:gen-snc}.
We may work locally on $X$ and assume that $X = \Spec R$
and $E = (V(f_1), \ldots, V(f_n))$ where $f_1, \ldots, f_n$
is a regular sequence in $R$.
Then $X_{r^{-1}E}$ is given by $[\Spec R'/A]$
with $R' = R[t_1, \ldots, t_n]/(t^{r_1}_1 - f_1, \cdots t^{r_n}_n - f_n)$
and $A = \rmu_{r_1}\times \cdots \times \rmu_{r_n}$ (cf.\
Example~\ref{ex:root-local-description}).
The divisor $r^{-1}E$ is given by $(V(t_1), \ldots, V(t_n))$.
It is now easy to verify that $t_1, \ldots, t_n$ is a
regular sequence in $R'$,
which proves the statement.

Now assume in addition that the ring $R/(f_1, \ldots, f_n)$
is regular.
Then the same holds for the rings $R'/(t_1, \ldots, t_r)$
for $0 \leq r \leq n$ since
$R'/(t_1, \ldots, t_n) \cong R/(f_1, \ldots, f_n)$
and $t_1, \ldots, t_n$ is a regular sequence.
In particular, this implies \ref{it:smooth} since
smoothness over $\kk$ is equivalent to regularity
after base change to an algebraically closed field.
\end{proof}

Due to
Proposition~\ref{prop:basic-iterated-root}.\ref{it:iterated-root}
root constructions in generalized snc divisors are 
sometimes referred to as \define{iterated root constructions}.


\section{Semiorthogonal decompositions for root stacks}
\label{sec:semiorthogonal}
In many aspects, root stacks behave like blowups.
For example, they give
rise to semiorthogonal decompositions.
This was observed by Ishii--Ueda in~\cite[Theorem~1.6]{iu2011}.
In this section, we reprove this theorem in a more general setting as
Theorem~\ref{thm:iishi-ueda}.
We also give an explicit,
combinatorial description of the semiorthogonal decomposition of
the derived 
category of an iterated root stack.

\begin{lemma}
[{cf.\ \cite[Theorem~4.14.(1)]{hr2014}}]
\label{lem:right-adjoint}
Let $\iota \colon E \to X$ be an effective Cartier divisor
on an algebraic stack $X$.
Then the functor 
$\iota_\ast =\R \iota_\ast \colon \Dqc(E) \to \Dqc(X)$
admits a right adjoint 
$\iota^\times$, 
and both $\iota_\ast$ and
$\iota^\times$ preserve perfect complexes.
\end{lemma}

Before we prove the lemma, we introduce some auxiliary notation.
Let $Z$ be an algebraic stack and $\mathcal{R}$ a quasi-coherent
sheaf of commutative $\mathcal{O}_Z$-algebras.
We call a sheaf of $\mathcal{R}$-modules quasi-coherent if it is
quasi-coherent as an $\mathcal{O}_Z$-module.
Let $\D(Z, \mathcal{R})$ denote the derived category of sheaves of
$\mathcal{R}$-modules in the topos $Z_\lisset$
and let $\Dqc(Z, \mathcal{R})$ be the full subcategory of objects
with 
quasi-coherent cohomology.
More generally, the definitions of $\D(Z, \mathcal{R})$ and
$\Dqc(Z, \mathcal{R})$ generalize in the obvious way to the case
that $\mathcal{R}$ is a quasi-coherent sheaf of commutative
dg $\mathcal{O}_Z$-algebras.

\begin{proof}
Consider the sheaf of commutative dg $\mathcal{O}_X$-algebras $\mathcal{R} = (\mathcal{O}_X(-E) \to
\mathcal{O}_X)$ where $\mathcal{O}_X$ sits in degree zero.
It comes with a quasi-isomorphism
$\mathcal{R}
\ra \iota_\ast\mathcal{O}_E$ of sheaves
of dg $\mathcal{O}_X$-algebras.  
The pushforward
\begin{equation*}
\iota_\ast \colon \Dqc(E) = \Dqc(E, \mathcal{O}_E) \to \Dqc(X) =
\Dqc(X, \mathcal{O}_X)  
\end{equation*}
factors as
\begin{equation}
  \label{eq:8}
  \Dqc(E, \mathcal{O}_E) 
  \xrightarrow{\sim} \Dqc(X, \iota_\ast\mathcal{O}_E) 
  \xrightarrow{\sim} \Dqc(X, \mathcal{R})
  \xrightarrow{\alpha_\ast} \Dqc(X, \mathcal{O}_X).
\end{equation}
Indeed, the first functor is an equivalence since $\iota$ is
affine  \cite[Corollary~2.7]{hr2014}.
The second equivalence is induced by restriction along the quasi-isomorphism
$\mathcal{R} \to \iota_\ast\mathcal{O}_E$ \cite[Proposition~1.5.6]{riche2010}.
The third functor $\alpha_\ast$ is induced by restriction along
the structure morphism $\mathcal{O}_X \to \mathcal{R}$.
This reduces the problem of finding a right adjoint to $\iota_\ast$
to finding a right adjoint to $\alpha_\ast$.

On the level of complexes,
the functor $\sheafHom_{\mathcal{O}_X}(\mathcal{R}, -)$
is easily seen to be right adjoint to restriction along
$\mathcal{O}_X \to \mathcal{R}$.
Since $\mathcal{R}$ is strictly perfect as a complex of $\mathcal{O}_X$-modules,
the functor $\sheafHom_{\mathcal{O}_X}(\mathcal{R}, -)$ takes acyclic complexes
to acyclic complexes and
descends to a
right adjoint $\alpha^\times \colon \Dqc(\mathcal{O}_X, X) \to
\Dqc(\mathcal{R}, X)$ of $\alpha_\ast$.
This proves the existence of $i^\times$.

Next we prove that $\iota_\ast$ preserves perfect complexes.
Let $\mathcal{F} \in \Dpf(E)$. 
The question whether $\iota_\ast\mathcal{F}$ is perfect
is local on $X$.
Since vector bundles on $E$ trivialize locally on $X$,
we may assume that $\mathcal{F}$ is a bounded complex of finite free modules.
But now the fact that $\iota_\ast\mathcal{O}_E$ is perfect implies
that $\iota_\ast\mathcal{F}$ is perfect.

Finally, let us prove that $\iota^\times$ preserves perfect complexes.
Since the question is local on $X$,
it is enough to verify that 
$\iota^\times(\mathcal{O}_X)$ is perfect.
Observe first that 
$\alpha^\times(\mathcal{O}_X)=\sheafHom_{\mathcal{O}_X}(\mathcal{R},
\mathcal{O}_X) 
\cong
\mathcal{R}\otimes_{\mathcal{O}_X} \mathcal{O}_X(E)[-1]$
in $\Dqc(X, \mathcal{R})$.
Under the equivalences of \eqref{eq:8}, this object corresponds to
the object 
$\iota_*\mathcal{O}_E \otimes_{\mathcal{O}_X}
\mathcal{O}_X(E)[-1]
\cong \iota_*( \mathcal{O}_E \otimes_{\mathcal{O}_E}
\iota^*(\mathcal{O}_X(E))[-1] \cong 
\iota_*(\iota^*(\mathcal{O}_X(E))[-1]$ of 
$\Dqc(X, \iota_*\mathcal{O}_E)$
and to the object 
$\iota^*(\mathcal{O}_X(E))[-1]$ of $\Dqc(E)$.
This latter object is obviously perfect and isomorphic to 
$\iota^\times(\mathcal{O}_X)$.
\end{proof}

\begin{lemma}
  \label{l:pull-push-triangle}
  In the setting of
  Lemma~\ref{lem:right-adjoint},
  given 
  any object $\mathcal{F}$ of $\Dqc(E)$, the adjunction counit
  $\L\iota^* \iota_* \mathcal{F}
  \ra 
  \mathcal{F}$
  fits into a triangle
  \begin{equation*}
    \mathcal{F} \otimes_{\mathcal{O}_E} \iota^*
    \mathcal{O}_X(-E)[1]
    \ra
    \L\iota^* \iota_* \mathcal{F}
    \ra 
    \mathcal{F}
    \ra
    \mathcal{F} \otimes_{\mathcal{O}_E} \iota^*
    \mathcal{O}_X(-E)[2]
  \end{equation*}
\end{lemma}

\begin{proof}
  Recall the factorization \eqref{eq:8} of $\iota_*$
  from the proof of Lemma~\ref{lem:right-adjoint}.
  The functor $\alpha_*$ occuring there is restriction of scalars
  along $\mathcal{O}_X \ra \mathcal{R}$. 
  Extension of scalars
  $(- \otimes_{\mathcal{O}_X} \mathcal{R})$ preserves acyclic
  complexes 
  and therefore defines a 
  left adjoint $(- \otimes_{\mathcal{O}_X} \mathcal{R})
  \colon \Dqc(X, 
  \mathcal{O}_X) \ra \Dqc(X, \mathcal{R})$ to $\alpha_*$.
  Modulo the two equivalences in \eqref{eq:8} this
  functor is isomorphic to $\L \iota^*$, and the adjunction
  counit
  $\L\iota^* \iota_*  \ra \id$
  corresponds to the adjunction counit
  $(- \otimes_{\mathcal{O}_X} \mathcal{R}) \ra \id$.

  Let $\mathcal{M}$ be a dg $\mathcal{R}$-module. 
  The adjunction counit $\mathcal{M} \otimes_{\mathcal{O}_X}
  \mathcal{R} \ra \mathcal{M}$ is given by multiplication. We
  denote its kernel by $\mathcal{K}$ and obtain a short exact
  sequence 
  \begin{equation*}
    \mathcal{K}
    \hra
    \mathcal{M} \otimes_{\mathcal{O}_X} \mathcal{R} 
    \sra 
    \mathcal{M} 
  \end{equation*}
  of dg $\mathcal{R}$-modules.
  As complexes of dg $\mathcal{O}_X$-modules, we have an
  obvious isomorphism
  \begin{equation*}
    \mathcal{K} \cong
    \mathcal{M} \otimes_{\mathcal{O}_X}
    \mathcal{O}_X(-E)[1].
  \end{equation*}
  Assume that $\mathcal{M}$ is obtained from a complex
  of dg $\iota_*\mathcal{O}_E$-modules by restriction along
  $\mathcal{R} \ra \iota_*\mathcal{O}_E$.
  Then this isomorphism is even an isomorphism of dg
  $\mathcal{R}$-modules. Since   
  $\mathcal{M} \otimes_{\mathcal{O}_X} \mathcal{O}_X(-E)$ 
  and
  $\mathcal{M} 
  \otimes_{\mathcal{R}} 
  (\mathcal{R}
  \otimes_{\mathcal{O}_X} \mathcal{O}_X(-E))$ are isomorphic as
  dg $\mathcal{R}$-modules, we obtain a triangle
  \begin{equation*}
    \mathcal{M} 
    \otimes_{\mathcal{R}} 
    (\mathcal{R}
    \otimes_{\mathcal{O}_X} \mathcal{O}_X(-E))[1]
    \ra
    \mathcal{M} \otimes_{\mathcal{O}_X} \mathcal{R} 
    \ra 
    \mathcal{M} 
    \ra
    \mathcal{M} 
    \otimes_{\mathcal{R}} 
    (\mathcal{R}
    \otimes_{\mathcal{O}_X} \mathcal{O}_X(-E))[2]
  \end{equation*}
  in $\D(X, \mathcal{R})$. 
  Since restriction of scalars
  $\Dqc(X, \iota_*\mathcal{O}_E) \ra \Dqc(X, \mathcal{R})$ is an
  equivalence, we obtain such a triangle in $\Dqc(X,
  \mathcal{R})$ for any object $\mathcal{M}$ of
  $\Dqc(X, \mathcal{R})$. The claim follows.
\end{proof}

\begin{lemma}
\label{lem:left-adjoint-la}
Let $f\colon X \to Y$ be a concentrated morphism of algebraic
stacks such that $\R f_\ast \colon \Dqc(X) \ra \Dqc(Y)$ preserves
perfect complexes. Then $\L f^\ast\colon \Dpf(Y) \to \Dpf(X)$ has a left adjoint $f_\times$ given by 
$$
f_\times\colon \Dpf(X) \ra \Dpf(Y), \quad 
\mathcal{F}
\mapsto
(\R f_\ast (\mathcal{F}^\vee))^\vee
$$
where $(-)^\vee$ denotes the dual
$\R \sheafHom_{\mathcal{O}_X}(-, \mathcal{O}_X)$ on 
$\Dpf(X)$, and similarly for 
$\Dpf(Y)$.
\end{lemma}

\begin{proof}
This statement is a formal consequence of the dual $(-)^\vee$ being
an involutive anti-equivalence 
which respects derived pullbacks.
Explicitly, for $\mathcal{F} \in \Dpf(X)$ and $\mathcal{G} \in \Dpf(Y)$
we have
\begin{align*}
\Hom_{\D(X)}(\mathcal{F}, \L f^*\mathcal{G}) 
& \cong \Hom_{\D(X)}((\L f^*\mathcal{G})^\vee, \mathcal{F}^\vee)\\
& \cong \Hom_{\D(X)}(\L f^* (\mathcal{G}^\vee), \mathcal{F}^\vee)\\ 
& \cong \Hom_{\D(Y)}(\mathcal{G}^\vee, \R f_*(\mathcal{F}^\vee))\\
& \cong \Hom_{\D(Y)}((\R f_*(\mathcal{F}^\vee))^\vee, (\mathcal{G}^\vee)^\vee)\\
& \cong \Hom_{\D(Y)}(f_\times \mathcal{F}, \mathcal{G}).
\end{align*}
\end{proof}

\begin{lemma}
\label{lem:left-adjoint-ff}
Let $f\colon X \to Y$ be a concentrated morphism of algebraic
stacks.
Then $\L f^\ast\colon \Dqc(Y) \to \Dqc(X)$ is full and faithful
if and only if 
the natural morphism $\mathcal{O}_Y
\to \R f_\ast \mathcal{O}_X$ is an isomorphism.
\end{lemma}

\begin{proof}
A left adjoint functor is full and faithful if and only if the
adjunction unit is an isomorphism. 
In particular, one implication is trivial.
For the other implication, assume that $\mathcal{O}_Y
\to \R f_\ast \mathcal{O}_X$ is an isomorphism.
Then the
projection formula
\cite[Corollary~4.12]{hr2014} 
gives $$
\mathcal{G} 
\sira \R f_\ast \mathcal{O}_X \otimes \mathcal{G} 
\sira \R f_\ast (\mathcal{O}_X \otimes \L f^\ast \mathcal{G})  
\sira \R f_\ast \L f^\ast \mathcal{G},
\qquad
\mathcal{G} \in \Dqc(Y).
$$
This shows that the adjunction unit is an isomorphism.
\end{proof}

\begin{lemma}
\label{l:tame-pushforward}
Let $X$ be a tame algebraic stack with finite inertia,
and let $\pi\colon X \to X_\coarse$ denote the canonical
morphism to its coarse space.
Then the natural morphism
$\mathcal{O}_{X_\coarse} \to \R\pi_\ast \mathcal{O}_X$
is an isomorphism.
Moreover, if $\pi$ is flat and of finite presentation,
then
$\R\pi_\ast \colon \Dqc(X) \to \Dqc(X_\coarse)$
preserves perfect complexes.
\end{lemma}
\begin{proof}
To check that a morphism in the derived category $\Dqc(X_\cs)$ is
an isomorphism, we may pass to an fppf covering by affine schemes. 
Since the derived push-forward and the formation of the coarse space
commute with flat base change,
we reduce to the situation where $X_\coarse$ is affine.
Since $\pi\colon X \to X_\coarse$ is the canonical morphism to the
coarse space, it is separated and quasi-compact.
Furthermore $X$ has finite stabilizers.
By Remark~\ref{rem:dqc} and~\ref{rem:dqc2},
we can therefore identify $\Dqc(X)$ with $\D(\Qcoh(X))$ and similarly
for $X_\coarse$.
The canonical morphism
$\mathcal{O}_{X_\coarse}\to \pi_\ast\mathcal{O}_X$
is an isomorphism,
again because $\pi$ is the structure morphism to the coarse space.
Hence the first statement follows from the exactness of
$\pi_\ast\colon \Qcoh(X) \to \Qcoh(X_\coarse)$
which is a consequence of the tameness hypothesis.

For the other statement, we may again work locally on $X_\coarse$,
since perfectness of a complex is a local property.
Hence we may again assume that $X_\coarse$ is affine.
We can also assume that we have a finite locally free covering
$\alpha\colon U \to X$ by an affine scheme $U$
(cf. \cite[Theorem~6.10 and~Proposition~6.11]{rydh2013}).

In this situation, we claim that the object $\alpha_\ast\mathcal{O}_U$ is a
compact projective generator for $\Qcoh(X)$.
In particular, the category $\Qcoh(X)$ is equivalent to the category of
modules for a not necessarily commutative ring.

Now we prove the statement claimed above.
First note that since
$\alpha$ is affine, the functor $\alpha_\ast$ has a right adjoint $\alpha^\times$
with the property that $\alpha_\ast\alpha^\times =
\sheafHom_{\mathcal{O}_X}(\mathcal{\alpha_\ast\mathcal{O}_U}, -)$.
Since $\alpha_\ast\mathcal{O}_U$ is finite locally free,
it follows that the functor $\alpha_\ast\alpha^\times$ is exact,
faithful, and commutes with filtered colimits.
Since $\alpha$ is affine, the functor $\alpha_\ast$ reflects these properties,
which implies that also $\alpha^\times$ is exact, faithful and commutes
with filtered colimits.
Finally, since $U$ is affine, it follows that also the functor
$$
\Hom_{\mathcal{O}_X}(\alpha_\ast\mathcal{O}_U, -)
\cong
\Hom_{\mathcal{O}_U}(\mathcal{O}_U, \alpha^\times(-))
$$
has these properties, so $\alpha_\ast\mathcal{O}_U$ is indeed a
compact, projective generator for $\Qcoh(X)$. 

It follows that the compact objects of the derived category $\Dqc(X)$,
which coincides with $\D(\Qcoh(X))$ by Remark~\ref{rem:dqc},
are precisely those isomorphic to bounded
complexes of compact projective objects.
By tameness, the perfect objects of $\Dqc(X)$ coincide with the
compact objects \cite[Remark~4.6]{hr2014}.
Hence it suffices to show that
$\pi_\ast\colon \Qcoh(X) \to \Qcoh(X_\coarse)$
preserves compact projective objects.

To prove this, we assume that $\mathcal{P}$ is a compact,
projective object in $\Qcoh(X)$.
Since $\alpha_\ast(\mathcal{O}_U)$ is a compact, projective generator
of $\Qcoh(X)$,
there exists a split surjection
$\alpha_\ast(\mathcal{O}_U)^{\oplus n} \to \mathcal{P}$
for some positive integer $n$.
Hence also the pushforward
$\pi_\ast\alpha_\ast(\mathcal{O}_U)^{\oplus n} \to \pi_\ast\mathcal{P}$
is a split surjection.
But $\pi_\ast\alpha_\ast(\mathcal{O}_U)$ is finite locally free,
and hence compact and projective in $\Qcoh(X_\coarse)$,
by our assumption that $\pi \colon X \to X_\coarse$ is flat and of finite presentation.
It follows that $\pi_\ast\mathcal{P}$ is compact and projective,
which concludes the proof.
\end{proof}

\begin{example}
\label{ex:fully-faithful}
We give some examples of concentrated morphisms $f\colon X \to
Y$ of algebraic stacks such that
the functor $\R f_\ast \colon \Dqc(X) \ra \Dqc(Y)$ preserves
perfect complexes 
and the natural morphism $\mathcal{O}_Y
\to \R f_\ast \mathcal{O}_X$ is an isomorphism
(cf.\ Lemmas~\ref {lem:left-adjoint-la} and
\ref{lem:left-adjoint-ff}).
Note that these properties are fppf local on $Y$.
Examples where $f$ is representable are
\begin{enumerate}
\item
  \label{enum:blowups} 
  blow-ups of smooth algebraic stacks over a field in a
  smooth locus; 
\item more generally, proper birational morphisms between smooth algebraic
stacks over a field;
\item
  \label{enum:proj-bundle}
  projective bundles.
\end{enumerate}
Examples where $f$ is not necessarily representable are
\begin{enumerate}[resume]
\item
\label{it:fully-faithful-mu}
the morphism $\rho$ in the root diagram~(\ref{eq:root-diagram});
\item
\label{it:fully-faithful-root}
the morphism $\pi$ in the root diagram~(\ref{eq:root-diagram}).
\end{enumerate} 
The last two items follow by applying Lemma~\ref{l:tame-pushforward} after an
appropriate base change, and using part
\ref{enum:pi-proper-fp-birat},
\ref{enum:rho},
\ref{enum:coarse}, and
\ref{enum:tame}
of
Proposition~\ref{prop:basic-root}.
\end{example}

\begin{theorem}
  \label{thm:iishi-ueda}
  Let $X$ be an algebraic stack and $E \subset X$ an effective
  Cartier divisor.
  Fix a positive integer $r$ and let
  $\pi \colon \widetilde{X} = X_{r\inv E} \ra X$ be the $r$-th root
  construction of $X$ in $E$ with $\iota$ and $\rho$
  as in the root diagram \eqref{eq:root-diagram}.
  Then the functors
  \begin{equation}
    \label{eq:pi}
    \pi^* \colon \Dpf(X) \ra \Dpf(\widetilde{X}),
  \end{equation}
  \begin{equation}
    \label{eq:Phi_a}
    \Phi_a :=
    \mathcal{O}_{\widetilde{X}}\left(ar^{-1}E\right) \otimes
    \iota_*\rho^*(-)
    \colon \Dpf(E) \ra \Dpf(\widetilde{X})
  \end{equation}
  for $a \in \{1, \ldots, r-1\}$,
  are full and faithful and admit left and right adjoints.
  Furthermore, the category $\D(\widetilde{X})$ has the
  semiorthogonal decomposition
  \begin{equation}
    \label{eq:semi-od-root-construction}
    \D(\widetilde{X})
    =\big\langle
    \Phi_{r-1},
    \dots,
    \Phi_{1},
    \pi^*
    \big\rangle
  \end{equation}
  into admissible subcategories.
\end{theorem}

Recall
that $\pi$ and $\rho$ are flat and that $\iota$ is the embedding
of the Cartier divisor $r^{-1}E$ (Proposition~\ref{prop:basic-root},
part
\ref{enum:cartier},
\ref{enum:pi-proper-fp-birat},
\ref{enum:rho})
and that $\mathcal{O}_{\widetilde{X}}(ar^{-1}E)$ is a
line bundle. Therefore we omitted the usual decorations for derived
functors in  
\eqref{eq:pi} and
\eqref{eq:Phi_a}.
Also note that $\Phi_a$ is well-defined by Lemma~\ref{lem:right-adjoint}.

\begin{proof}[Proof of Theorem~\ref{thm:iishi-ueda}]
Both functors 
$\rho^\ast \colon \Dpf(E) \ra \Dpf(r^{-1}E)$ and
$\pi^\ast \colon \Dpf(X) \ra \Dpf(\widetilde{X})$ are
full and faithful and admit left and right adjoints 
by part \ref{it:fully-faithful-mu} and
\ref{it:fully-faithful-root}
of
Example~\ref{ex:fully-faithful} and
Lemmas~\ref{lem:left-adjoint-la} and \ref{lem:left-adjoint-ff}.
Since tensoring with a line bundle is an autoequivalence
and since $\iota_* \colon \Dpf(r^{-1}E) \ra \Dpf(\widetilde{X})$
admits left and right adjoints, by Lemma~\ref{lem:right-adjoint},
we
deduce that the functors $\Phi_a$ admit left and right adjoints.

The stack $r^{-1}E$ is a $\rmu_r$-gerbe over $E$.
Therefore the category $\Mod(r^{-1}E)$ splits as
a direct sum
$
\bigoplus_{\chi = 0}^{r-1} \Mod(r^{-1}E)_\chi
$
according to the characters of the inertial action.
This induces a corresponding decomposition
$
\bigoplus_{\chi = 0}^{r-1} \Dpf(r^{-1}E)_\chi
$
of the triangulated category $\Dpf(r^{-1}E)$. The essential image
of $\rho^*$ is $\Dpf(r^{-1}E)_0$.

Consider 
\begin{equation*}
\Hom_{\D(\widetilde{X})}(\iota_\ast\mathcal{F}, \iota_\ast\mathcal{G})
\cong
\Hom_{\D(r^{-1}E)}(\L\iota^\ast\iota_\ast\mathcal{F}, \mathcal{G})
\label{eq:root-hom}
\end{equation*}
for $\mathcal{F}$, $\mathcal{G} \in \Dpf(r^{-1}E)$.
Since $\iota$ is the inclusion of an effective Cartier divisor,
Lemma~\ref{l:pull-push-triangle} provides a triangle
\begin{equation*}
  \mathcal{F} \otimes
  \mathcal{N}[1]
  \ra
  \L\iota^* \iota_* \mathcal{F}
  \ra 
  \mathcal{F}
  \ra
  \mathcal{F} \otimes
  \mathcal{N}[2]
\end{equation*}
where 
$\mathcal{N} = \iota^\ast
\mathcal{O}_{\widetilde{X}}(-r^{-1}E)$
is the conormal line bundle
of the closed immersion $\iota$.

Now assume that $\mathcal{F} \in \Dpf(r^{-1}E)_\chi$ and
$\mathcal{G} \in \Dpf(r^{-1}E)_\psi$.
Since $\mathcal{N} \in \Dpf(r^{-1}E)_1$, the above triangle
enables us to compute
\begin{equation}
\label{eq:hom-divisor}
\Hom_{\D(\widetilde{X})}(\iota_\ast\mathcal{F}, \iota_\ast\mathcal{G})
\cong
\left\{
\begin{array}{ll}
\Hom_{\D(r^{-1}E)}(\mathcal{F}, \mathcal{G}) & \text{if } \chi = \psi,\\
\Hom_{\D(r^{-1}E)}(\mathcal{F}[1]\otimes \mathcal{N}, \mathcal{G}) & \text{if 
$\chi+1=\psi$ in $\DZ/r$,}\\
0 & \text{otherwise.}\\
\end{array}
\right.
\end{equation}
In particular, we see that the restriction of the functor
$\iota_\ast$ to the category $\Dpf(r^{-1}E)_\chi$ is full and
faithful for each $\chi$. As a consequence, all functors $\Phi_a$
are full and faithful.

Moreover, given $\mathcal{H} \in \Dpf(X)$, we have
\begin{equation*}
\label{eq:hom-main}
\Hom_{\D(\widetilde{X})}(\pi^\ast\mathcal{H}, \iota_\ast\mathcal{G})
\cong
\Hom_{\D(r^{-1}E)}(\L\iota^\ast\pi^\ast\mathcal{H}, \mathcal{G})
\cong
\Hom_{\D(r^{-1}E)}(\rho^\ast\L\kappa^\ast\mathcal{H}, \mathcal{G}),
\end{equation*}
which vanishes if $\psi \neq 0$
since the essential image of $\rho^*$ is $\Dpf(r^{-1}E)_0$.

This, together with the third equality in~(\ref{eq:hom-divisor})
shows that
\begin{equation}
\iota_\ast \Dpf(r^{-1}E)_1,
\ldots,
\iota_\ast \Dpf(r^{-1}E)_{r-1},
\pi^\ast \Dpf(X)
\end{equation}
is a semiorthogonal sequence.
The projection formula \cite[Corollary~4.12]{hr2014} shows that
$\Phi_a \cong \iota_*(\mathcal{N}^{\otimes(-a)} \otimes \rho^*(-))$.
Hence the essential image of $\Phi_a$ lies in $\iota_*
\Dpf(r^{-1}E)_{-a}$ and the 
essential images of the functors in
\eqref{eq:semi-od-root-construction} form a semiorthogonal
sequence. 

Let $\mathcal{T}$ denote the smallest strict full triangulated subcategory
of $\Dpf(r^{-1}E)$ which contains all these essential
images. Then
\begin{equation*}
  \mathcal{T}= 
  \big\langle
  \Phi_{r-1},
  \dots,
  \Phi_{1},
  \pi^*
  \big\rangle
\end{equation*}
is a semiorthogonal decomposition into admissible subcategories.
It remains to prove that $\mathcal{T} = \Dpf(X_{r^{-1}E})$.
This can be done fppf locally on $X$ by conservative descent~\cite{bs2016}.
Hence we may work with the local description given in 
Example~\ref{ex:root-local-description}.
Using the notation from the example,
the category $\Qcoh(X_{r^{-1}E})$
is equivalent to the category of
$\mathbb{Z}/r\mathbb{Z}$-graded $R'$-modules.
We use the symbol $\langle -\rangle$ to denote
shifts with respect to the $\mathbb{Z}/r\mathbb{Z}$-grading.
More precisely, given a graded $R'$-module $M=\bigoplus M^n$,
we write $M\langle i\rangle$ for the graded $R'$-module with
components $(M \langle i\rangle)^n=M^{i+n}$.

Note that $P = R'\langle 0 \rangle \oplus \cdots \oplus R'\langle r-1 \rangle$
is a compact projective generator of $\Qcoh(X_{r^{-1}E})$.
This implies that $P$ is a classical generator of $\Dpf(X_{r^{-1}E})$.
Since each of the semiorthogonal summands of $\mathcal{T}$ is idempotent complete,
the same holds for $\mathcal{T}$.
Therefore, it is enough to prove that $R'\langle i\rangle$ is contained in $\mathcal{T}$
for each~$i$.
But $\mathcal{T}$ contains
$
\pi^\ast \mathcal{O}_X = R'\langle 0\rangle$,
and
$
\Phi_{i} \mathcal{O}_E = R'/(t)\langle i\rangle=R/(f)\langle i\rangle
$
for $i \in \{1, \ldots r-1\}$,
so this follows from the triangles
$$
R'\langle i-1\rangle \xra{t}
R'\langle i\rangle \to
R/(f)\langle i\rangle \to
R'\langle i-1\rangle[1] 
$$
and induction on $i$ starting with $i = 1$. 
\end{proof}

\begin{remark}
In \cite{iu2011}, Ishii and Ueda state
Theorem~\ref{thm:iishi-ueda}
for bounded derived categories of coherent sheaves.
They assume that $X$ and $E$ are quasi-compact, separated
Deligne--Mumford stacks which are smooth over $\mathbb{C}$
(although not all of these conditions are explicitly mentioned).
Under these hypotheses the triangulated categories
$D^\bd(\Coh(X))$, $D^\bd(\Coh(E))$,
and~$D^\bd(\Coh(X_{r^{-1}E}))$ are equivalent to the
categories $\Dpf(X)$, $\Dpf(E)$,
and~$\Dpf(X_{r^{-1}E})$ respectively (cf.~Remark~\ref{rem:D-Coh-and-DbCoh=Dpf}).
\end{remark}


Next, we generalize Theorem~\ref{thm:iishi-ueda} to
iterated root stacks.

\begin{theorem}
\label{thm:semi-orthogonal-iterated}
Let $X$ be an algebraic stack and $E$ a generalized snc divisor
on $X$ with components indexed by $I$.
Fix a multi-index $r > 0$ in $\DZ^I$ and let 
$X_{r^{-1}E}$ be the $r$-th root stack as
in Definition~\ref{def:iterated-root}.
For any multi-index $a$ satisfying $r > a \geq 0$
denote its support by $I_a \subseteq I$ and 
consider the diagram
\begin{equation}
\label{eq:iterated-root-diagram}
\xymatrix{
r^{-1}E(I_a) \ar[r]^-{\iota_a} \ar[d]_-{\rho_a}
& 
{X_{r^{-1}E}} \ar[d]^-{\pi}
\\
{E(I_a)} \ar[r]
&
{X,}
}
\end{equation}
where $E(I_a):= \cap_{i \in I_a} E_i$ and
$r^{-1}E(I_a):= \cap_{i \in I_a} (r^{-1}E)_i$.
Then all the functors
\begin{equation}
\label{eq:transform-functor}
\Phi_a := 
\mathcal{O}_{X_{r^{-1}E}}(ar^{-1}E)
\otimes
(\iota_a)_\ast\rho^\ast_a(-)
\colon \Dpf(E(I_a)) \to \Dpf(X_{r^{-1}E})
\end{equation}
are full and faithful and admit left and right adjoints.
Furthermore, the category $\Dpf(X_{r^{-1}E})$ has the semiorthogonal
decomposition
\begin{equation}
\label{eq:iterated-decomposition}
\langle \Phi_{a} \mid r > a \geq 0\rangle
\end{equation}
into admissible subcategories.
Here the multi-indexes $a$ with $r > a \geq 0$ are
arranged into any sequence $a^{(1)}, a^{(2)}, \dots, a^{(m)}$ such that $a^{(s)} \geq a^{(t)}$ implies $s \leq t$ for
all $s, t \in \{1, \dots, m\}$ where $m=\prod_{i \in I} r_i$.
\end{theorem}

\begin{remark}
  \label{rmk:iterated-root-special}
  If $r$ has at most one coordinate which is
  strictly bigger
  than one,
  then the root stack $X_{r^{-1}E}$ is isomorphic to a non-iterated
  root stack,
  and  we recover Theorem~\ref{thm:iishi-ueda}.
\end{remark}

\begin{example}
  \label{exam:T-semiorthog-decom-two-divisors}
  If our generalized snc divisor $E$ has two components $E_1=D$,
  $E_2=F$ and  
  $r_1=4$ and $r_2=3$, the Hasse diagram of the poset $\{a \in \DZ^2
  \mid r > a \geq 0\}$
  (with arrows pointing to smaller elements) looks as follows.
  \begin{equation}
    \xymatrix@ur@=3ex{
      \ar[d] 
      {(3,2)}_{D \cap F} \ar[r] & \ar[d]
      {(2,2)}_{D \cap F} \ar[r] & \ar[d]
      {(1,2)}_{D \cap F} \ar[r] & \ar[d]
      {(0,2)_{F}} 
      \\
      \ar[d] 
      {(3,1)}_{D \cap F} \ar[r] & \ar[d]
      {(2,1)}_{D \cap F} \ar[r] & \ar[d]
      {(1,1)}_{D \cap F} \ar[r] & \ar[d]
      {(0,1)_{F}} 
      \\
      {(3,0)_{D}} \ar[r] &
      {(2,0)_{D}} \ar[r] &
      {(1,0)_{D}} \ar[r] &
      {(0,0)_{X}} 
    }
  \end{equation}
  The index at a vertex $a=(a_1,a_2)$ is $E(I_a)$.
  If we think of such a vertex as representing the 
  essential image of $\Dpf(E(I_a))$
  under the fully faithful functor $\Phi_a$, this gives a nice
  way to visualize the 
  semiorthogonal decompositions
  \eqref{thm:semi-orthogonal-iterated}
  for all allowed sequences $a^{(1)}, \dots, a^{(12)}$ at once.
  If there is a nonzero morphism from an object of the category
  represented by a vertex $a$ to an object of the category
  represented by a vertex $b$, then there is a directed path from
  $a$ to $b$ in the Hasse diagram.
  Of course, we could have used the concept of a semiorthogonal
  decomposition indexed by a poset.
\end{example}

\begin{proof}[Proof of Theorem~\ref{thm:semi-orthogonal-iterated}]
We use induction on the number of coordinates of $r$ which are
strictly bigger than one.
In light of Remark~\ref{rmk:iterated-root-special},
we may assume that $r_{i_0} > 1$ for some index $i_0 \in
I$ which we fix.
Then the multi-index $r$ factors as $s\cdot t$,
where $t = (t_i)$ satisfies $t_{i_0} = r_{i_0}$ 
and $t_i = 1$ for all $i \not= i_0$.
By Proposition~\ref{prop:basic-iterated-root}.\ref{it:iterated-root},
the root stack $X_{r^{-1}E} \to X$ decomposes into a sequence
\begin{equation}
\label{eq:root-sequence}
X_{r^{-1}E} \to X_{s^{-1}E} \to X
\end{equation}
of root stacks.

Any multi-index $r> a \geq 0$ can be uniquely written as a sum
$a = a' + a''$ with $t > a' \geq 0$ and $s > a'' \geq 0$,
and this gives a bijective correspondence between the set of
multi-indexes $a$ with $r > a \geq 0$ and pairs $(a',a'')$ of
multi-indexes satisfying $t > a' \geq 0$ and $s > a'' \geq 0$.
The support $I_{a'}$ of such a multi-index $a'$ is contained
in $\{i_0\}$. 

For any such multi-index $a=a'+a''$ the sequence
\eqref{eq:root-sequence}
induces the decomposition
\begin{equation}
\label{eq:decomposed-transform}
\xymatrix{
r^{-1}E(I_a) \ar[r]\ar[d]
  & r^{-1}E(I_{a'}) \ar[r]\ar[d]
  & X_{r^{-1}E}\ar[d]
\\
s^{-1}E(I_a) \ar[r]\ar[d]
  & s^{-1}E(I_{a'}) \ar[r]\ar[d]
  & X_{s^{-1}E}\ar[d]
\\
E(I_a) \ar[r]
  & E(I_{a'}) \ar[r]
  & X
\\
}
\end{equation}
of the diagram \eqref{eq:iterated-root-diagram}.

In the rest of this proof, we call a diagram of the form
\eqref{eq:iterated-root-diagram} a transform diagram.
The upper right square in \eqref{eq:decomposed-transform} depends
on the support of $a'$ (but not on $a$ and $a''$) and 
is a transform diagram for the $t$-th root of the divisor
$s^{-1}E$ in $X_{s^{-1}E}$. 
If $I_{a'}=\{i_0\}$ it is, in fact, 
by the special form of $t$, 
a root diagram 
for the $r_{i_0}$-th root of the effective Cartier divisor
$(s^{-1}E)_{i_0}$ on $X_{s^{-1}E}$.
Denote the functors corresponding to \eqref{eq:transform-functor}
for this root construction by
$\Phi'_{a'}\colon \Dpf(s^{-1}E(I_{a'})) \to \Dpf(X_{r^{-1}E})$.
Then Theorem~\ref{thm:iishi-ueda} yields the semiorthogonal
decomposition 
\begin{equation}
\label{eq:semiorthogonal-a'}
\Dpf(X_{r^{-1}}E)=\langle\Phi'_{a'} \mid t > a' \geq 0\rangle.
\end{equation}

Fix $t > a' \geq 0$ for a moment.
Note that the lower right square of \eqref{eq:decomposed-transform}
is cartesian. This is trivial if $I_{a'}=\emptyset$ and otherwise
follows from the fact that $s_{i_0} = 1$.
As a consequence, $s^{-1}E(I_{a'}) \to E(I_{a'})$ is an $s'$-th
root of the generalized snc divisor 
$E'= (E_i \cap E(I_{a'}))_{i \in I -\{i_0\}}$
where $s'=(s_i)_{i \in I -\{i_0\}}$. 
The corresponding transform diagram for $s > a'' \geq 0$ is the
lower left square of
diagram \eqref{eq:decomposed-transform} where $a=a'+a''$.
Let $\Phi''_{a''}\colon \Dpf(E(I_{a})) \to \Dpf(s^{-1}E(I_{a'}))$
denote the functors corresponding to \eqref{eq:transform-functor}
for this iterated root construction; here $a''$ is identified with its
restriction to $I - \{i_0\}$.
By the induction hypothesis,
we obtain the semiorthogonal decomposition 
\begin{equation}
  \label{eq:semiorthogonal-a''}
  \Dpf(s^{-1}E(I_{a'}))=\langle\Phi''_{a''} \mid s > a'' \geq 0\rangle.
\end{equation}

Combining the decompositions 
\eqref{eq:semiorthogonal-a'}
and \eqref{eq:semiorthogonal-a''}
yields the
semiorthogonal decomposition   
\begin{equation}
\label{eq:outer-semiorthogonal}
\Dpf(X_{r^{-1}}E)=
\langle\Phi'_{a'}\circ\Phi''_{a''} \mid
t > a' \geq 0,\ s > a'' \geq 0\rangle.
\end{equation}

Next, we establish an isomorphism $\Phi_{a} \cong
\Phi'_{a'}\circ\Phi''_{a''}$ for $a=a'+a''$ as above.
Let
$$
\mathcal{L}' 
=\mathcal{O}_{X_{r^{-1}E}}(a'r^{-1}E),
\qquad
\mathcal{L}'' 
=\mathcal{O}_{X_{s^{-1}E}}(a''s^{-1}E).
$$
Furthermore, we denote the horizontal arrows in \eqref{eq:decomposed-transform}
by $\iota_{ij}$ and the vertical arrows by $\rho_{ij}$,
where $i$ denotes the row and $j$ the column of domain of the morphism
as viewed in the diagram.
With this notation, we have identities
$$
\Phi'_{a'} =
\mathcal{L}'\otimes (\iota_{12})_\ast \rho_{12}^\ast(-),
\qquad
\Phi''_{a''} =
\iota_{22}^\ast\mathcal{L}''\otimes (\iota_{21})_\ast\rho_{21}^\ast(-).
$$
Consider the composition $\Phi'_{a'} \circ \Phi'_{a''}$ of these
two functors. 
By the projection formula for $\iota_{12}$
(\cite[Corollary~4.12]{hr2014})
and the fact that pullbacks and tensor products commute,
we see that this composition
is isomorphic to
\begin{equation}
\label{eq:composed-functors}
\mathcal{L}' \otimes
\rho_{13}^\ast\mathcal{L}'' \otimes 
(\iota_{12})_\ast\rho_{12}^\ast
(\iota_{21})_\ast\rho_{21}^\ast(-).
\end{equation}
Since the upper left square in \eqref{eq:decomposed-transform} is
cartesian, flat base change (\cite[Theorem~2.6.(4)]{hr2014})
along the flat morphism $\rho_{12}$ (Proposition~\ref{prop:basic-root}.\ref{enum:rho})
shows that our composition \eqref{eq:composed-functors} is
isomorphic to 
\begin{equation}
\label{eq:composed-functors-2}
\mathcal{L}'\otimes
\rho_{13}^\ast\mathcal{L}''\otimes 
(\iota_{12}\circ\iota_{11})_\ast
(\rho_{11}\circ\rho_{21})^\ast.
\end{equation}
Now
$$
\mathcal{L}'\otimes \rho_{13}^\ast\mathcal{L}'' = 
\mathcal{O}_{X_{r^{-1}E}}\left((a' + ta'')r^{-1}E\right). 
$$
But $a' + ta'' = a'+a''=a$ since $t_i = 1$ for $i$ in the support
of $a''$, 
so \eqref{eq:composed-functors-2} is indeed isomorphic
to~$\Phi_{a}$. 
This shows
$\Phi_{a} \cong
\Phi'_{a'}\circ\Phi''_{a''}$.

Hence
$\Phi_{a}$ is full and faithful and
admits left and right adjoints, and 
the semiorthogonal decomposition
\eqref{eq:outer-semiorthogonal} simplifies to
\begin{equation}
\label{eq:outer-semiorthogonal-2}
\Dpf(X_{r^{-1}}E)=
\langle\Phi_{a} \mid t > a' \geq 0,\ s > a'' \geq 0,\ a = a' + a''\rangle.
\end{equation}

Since $i_0$ was arbitrary with
$r_{i_0}>1$ the above shows: if $a$ and $b$ 
are two multi-indexes
with $r>a \geq 0$
and $r > b \geq 0$ such that a nonzero morphism from an object of the
essential image of $\Phi_a$ to an object of the essential image
of $\Phi_b$ exists, then $a \geq b$. This proves the theorem.
\end{proof}


\section{Differential graded enhancements and geometricity}
\label{sec:diff-grad-enhancements}

Many triangulated categories are homotopy
categories of certain differential graded (dg) categories.
This observation leads to the notion of a dg enhancement of a
triangulated category. We introduce obvious dg
enhancements of the derived categories considered in this
article and explain how to lift
certain derived functors to dg functors between these
enhancements. 
We then recall Orlov's notion of a
geometric dg category and state his main glueing result.

We assume that the reader has some familiarity with 
differential graded categories, see for example
\cite{keller2006, toen2008}.
In this section, we will work over a fixed field $\kk$
and assume that all our triangulated categories and
all our dg~categories are $\kk$-linear.

\subsection{DG enhancements}
\label{sec:dg-texorpdfstr-enhan}
We introduce the dg enhancements we will use in the rest of this
article.

The homotopy category of a dg~category $\mathcal{A}$
is denoted by $[\mathcal{A}]$.
Recall that if $\mathcal{A}$ is a \define{pretriangulated}
dg~category, then the homotopy category $[\mathcal{A}]$
has a canonical structure of a triangulated category.

\begin{definition}
  \label{d:enhancement}
  A \define{dg~enhancement} of a triangulated
  category $\mathcal{T}$ is a pair $(\mathcal{E},
  \epsilon)$ consisting of a 
  pretriangulated dg~category $\mathcal{E}$ together with an equivalence
  $\epsilon \colon [\mathcal{E}] \sira \mathcal{T}$ of
  triangulated categories.   
\end{definition}

\begin{example}
  \label{expl:k-ringed-topos-enhancement}
  Let $(X, \mathcal{O})$ be a ringed topos over $\kk$.
  In the dg~category of complexes of
  $\mathcal{O}$-modules, consider the 
  full dg~subcategory $\Ddg(X, \mathcal{O})$ 
  consisting of h-injective complexes of injective
  $\mathcal{O}$-modules.
  This pretriangulated dg~category together
  with the obvious equivalence
  \begin{equation}
    \label{eq:1}
    [\Ddg(X, \mathcal{O})] \sira \D(X, \mathcal{O})
  \end{equation}
  forms a dg~enhancement of $\D(X)$. 
  We chose to work with these
  dg~enhancements in this article.
\end{example}

\begin{remark}
  \label{rem:drinfeld-quotient}
  Another dg enhancement of $\D(X,\mathcal{O})$ is provided by
  the Drinfeld dg quotient of the dg category of complexes of
  $\mathcal{O}$-modules by its full dg subcategory of acyclic
  complexes.
\end{remark}

\begin{remark}
  \label{rem:induced-enhancement}
  If $(\mathcal{E}, \epsilon)$ is a dg enhancement of
  $\mathcal{T}$ then any strict full triangulated subcategory
  $\mathcal{S}$ of $\mathcal{T}$ has an induced dg
  enhancement: just take the full dg subcategory of
  $\mathcal{E}$ of objects that go to objects of
  $\mathcal{S}$ under $\epsilon$, and restrict $\epsilon$
  appropriately. 
\end{remark}

\begin{example}
\label{expl:algebraic-stack-enhancement}
Let $X$ be an algebraic stack over $\kk$ and consider the ringed
topos $(X_\liset, \mathcal{O}_X)$.
The derived category $\D(X) = \D(X_\liset, \mathcal{O}_X)$ has
the dg~enhancement $\Ddg(X) := \Ddg(X_\liset, \mathcal{O}_X)$.
By Remark~\ref{rem:induced-enhancement},
the strict triangulated subcategories 
$\Dqc(X)$ and $\Dpf(X)$ have induced
dg~enhancements which we denote by
$\Ddgqc(X)$ and $\Ddgpf(X)$, respectively. 
\end{example}

\begin{example}
\label{expl:DM-stack-enhancement}
If $X$ is a Deligne--Mumford stack over $\kk$ we could instead consider
the ringed topos $(X_\et, \mathcal{O}_X)$ and define
the dg~enhancements $\Ddg(X_\et)$, $\Ddgqc(X_\et)$
and~$\Ddgpf(X_\et)$ 
for the triangulated categories
$\D(X_\et)$, $\Dqc(X_\et)$ and~$\Dpf(X_\et)$
in a similar way as in the previous example.
\end{example}

\begin{remark}
  \label{rem:DM-compare-liset-small-etale-dg}
  Let $X$ be a Deligne--Mumford stack.
As described in Remark~\ref{rem:DM-compare-liset-small-etale},
the morphism \eqref{eq:7} of ringed topoi
induces a triangulated equivalence
  $\epsilon_\ast \colon \D_\qc(X_\liset) \sira
  \D_\qc(X_\et)$.
Since $\epsilon^*$ is exact,
the functor  
  $\epsilon_* \colon
  \Mod(X_\liset, \mathcal{O}) \ra 
  \Mod(X_\et, \mathcal{O})$ 
  preserves injectives and h-injective complexes.
  Therefore we obtain quasi-equivalences 
  $\Ddgqc(X_\liset) \ra \Ddgqc(X_\et)$
  and
  $\Ddgpf(X_\liset) \ra \Ddgpf(X_\et)$
  lifting the equivalences
  $\Dqc(X_\liset) \sira \Dqc(X_\et)$
  and
  $\Dpf(X_\liset) \sira \Dpf(X_\et)$
  to dg~enhancements.
\end{remark}

\subsection{Uniqueness of dg enhancements}
\label{sec:uniq-dg-enhanc}

We would like to point out that the derived categories we are
mainly interested in have unique dg enhancements in the sense of
the following definition.

\begin{definition}
  [{cf.\ \cite{lunts-orlov-enhancement,
      canonaco-stellari-uniqueness-of-dg-enhancements}}] 
  \label{d:unique-enhancement}
  We say that a triangulated category $\mathcal{T}$
  \define{has a unique dg enhancement} if it has a dg
  enhancement and given any two dg enhancements 
  $(\mathcal{E}, \epsilon)$ 
  and 
  $(\mathcal{E}', \epsilon')$ 
  of $\mathcal{T}$, the dg~categories
$\mathcal{E}$ and $\mathcal{E}'$ are
quasi-equivalent. That is, they are connected by a zig-zag of
quasi-equivalences.
\end{definition}

\begin{remark}
\label{rem:Dpf-unique-enhancemenet}
By \cite[Proposition~6.10]{canonaco-stellari-uniqueness-of-dg-enhancements},
the derived category $\Dpf(X)$ for any separated, tame
algebraic stack $X$ which is smooth and of finite type over $\kk$ has a unique
dg enhancement.
In particular, this includes the stacks considered in the main theorems of this
article.
Indeed, Proposition~6.10 from \emph{loc.\ cit.}\ applies
since every coherent $\mathcal{O}_X$-module on $X$ is perfect by
the assumption that $X$ is regular (cf.\
Proposition~\ref{p:regular-DbCoh=Dpf}).
Furthermore, the category $\Qcoh(X)$ is generated by a set of
coherent $\mathcal{O}_X$-modules
as a Grothendieck category since every quasi-coherent
$\mathcal{O}_X$-module is the filtered colimit of its coherent submodules~\cite[Proposition~15.4]{lmb2000}.
\end{remark}

\subsection{Lifts of some derived functors to dg enhancements}
\label{sec:lifts-deriv-funct}

We need to lift some derived functors to the level of dg
enhancements.
Since our main results concern algebraic stacks over a
field we chose to use the methods of
\cite{olaf-six-operations-on-dg-enhancements-v1}.
We briefly recall the results we need.

If $(X, \mathcal{O})$ is a ringed topos over the field $\kk$,
we have replacement dg functors $\ii$ and $\ee$ on
the dg category of complexes of $\mathcal{O}$-modules.
The functor $\ii$ replaces a complex with a quasi-isomorphic
h-injective complex of injective $\mathcal{O}$-modules,
and $\ee$ replaces a complex with an h-flat complex of
flat $\mathcal{O}$-modules (\cite[Theorem~4.17]{olaf-six-operations-on-dg-enhancements-v1}).

Let $f \colon (X, \mathcal{O}) \ra (Y, \mathcal{O}')$ be a
morphism of ringed topoi over $\kk$.
Then the dg~functors
\begin{align*}
  \ulu{f} := \ii f^\ast\ee \colon & \Ddg(Y) \ra \Ddg(X),\\
  \ull{f} := \ii f_\ast \colon & \Ddg(X) \ra \Ddg(Y)  
\end{align*}
make the diagrams
\begin{equation}
  \label{eq:2}
  \xymatrix{
    [\Ddg(Y)] 
    \ar[rr]^-{[\ulu{f}]} 
    \ar[d]^-{\sim} 
    &&
    {[\Ddg(X)]} 
    \ar[d]^-{\sim} 
    \\
    {\D(Y)} \ar[rr]^-{\L f^*}
    &&
    {\D(X),}
  }
  \quad \quad \quad
  \xymatrix{
    {[\Ddg(X)]}
    \ar[rr]^-{[\ull{f}]}
    \ar[d]^-{\sim}
    &&
    {[\Ddg(Y)]}
    \ar[d]^-{\sim}
    \\
    {\D(X)} 
    \ar[rr]^-{\R f_*} 
    &&
    {\D(Y)} 
  }
\end{equation}
commutative up to isomorphisms of triangulated functors, by
\cite[Proposition~6.5]{olaf-six-operations-on-dg-enhancements-v1}.
The vertical arrows in these diagrams are given by the functor
\eqref{eq:1}. Here we abbreviate $\D(X)=\D(X,\mathcal{O})$ and
$\Ddg(X)=\Ddg(X, \mathcal{O})$ to ease the notation, and similarly for $Y$.

Similarly, if $E \in \D(X)$ is any object, the dg~functor 
\begin{equation*}
  (E \ul{\otimes} -) := \ii(E \otimes \ee(-)) \colon \Ddg(X) \ra
  \Ddg(X) 
\end{equation*}
makes the diagram
\begin{equation}
  \label{eq:3}
  \xymatrix{
    [\Ddg(X)] 
    \ar[rr]^-{[E \ul{\otimes} -]} 
    \ar[d]^-{\sim} 
    &&
    {[\Ddg(X)]} 
    \ar[d]^-{\sim} 
    \\
    {\D(X)} \ar[rr]^-{E \otimes^\L -}
    &&
    {\D(X)}
  }
\end{equation}
commutative up to an isomorphism of triangulated functors;
this follows immediately from
\cite[Section~6.3]{olaf-six-operations-on-dg-enhancements-v1}.

These three diagrams say that the dg functors
$\ulu{f}$, $\ull{f}$ and~$(E \ul{\otimes} -)$ lift the triangulated
functors $\L f^*$, $\R f_*$ and~$(E \otimes^\L-)$ to dg~enhancements.

\begin{remark}
  \label{rem:dg-lifts-subcategories}
  In the above situation assume that $\D_\diamond(X)$ and
  $\D_\diamond(Y)$ are strict 
  triangulated subcategories of $\D(X)$ and $\D(Y)$,
  respectively. 
  By 
  Remark~\ref{rem:induced-enhancement},
  these subcategories have induced dg enhancements
  $\Ddgdiamond(X)$ and $\Ddgdiamond(Y)$.
  If 
  $\L f^* \colon \D(Y) \ra \D(X)$ maps 
  $\D_\diamond(Y)$ to  $\D_\diamond(X)$, then $\ulu{f} \colon
  \Ddg(Y) \ra \Ddg(X)$ maps $\Ddgdiamond(Y)$ to
  $\Ddgdiamond(X)$, and the induced dg functor
  $\ulu{f} \colon \Ddgdiamond(Y) \ra \Ddgdiamond(X)$
  lifts the induced triangulated functor 
  $\L f^* \colon \D_\diamond(Y) \ra \D_\diamond(X)$: 
  diagram \eqref{eq:2} restricts to 
  \begin{equation}
    \label{eq:4}
    \xymatrix{
      [\Ddgdiamond(Y)] 
      \ar[rr]^-{[\ulu{f}]} 
      \ar[d]^-{\sim} 
      &&
      {[\Ddgdiamond(X)]} 
      \ar[d]^-{\sim} 
      \\
      {\D_\diamond(Y)} \ar[rr]^-{\L f^*}
      &&
      {\D_\diamond(X).}
    }
  \end{equation}
  Similar remarks apply to the functors $\R f_*$ and $(E
  \otimes^\L -)$. 
\end{remark}

\begin{example}
  \label{expl:DM-stack-dg-lifts}
  Let $f \colon X \ra Y$ be a concentrated morphism of
  Deligne--Mumford stacks over the field $\kk$.
  As stated in Remark~\ref{rem:left-derived-et},
  we get an induced morphism of ringed topoi.
  Hence Remark~\ref{rem:dg-lifts-subcategories} applies
  and the functors $\L f^*$ and $\R f_*$ lift to dg~functors 
  $\ulu{f} \colon \Ddgqc(Y_\et) \ra \Ddgqc(X_\et)$ and 
  $\ull{f} \colon \Ddgqc(X_\et) \ra \Ddgqc(Y_\et)$
  between the dg enhancements of
  Example~\ref{expl:DM-stack-enhancement}.
  If $E \in \Dqc(X_\et)$ then 
  $(E \otimes^\L -) \colon \Dqc(X_\et) \ra \Dqc(X_\et)$ lifts to
  a dg functor $(E \ul{\otimes} -) \colon \Ddgqc(X_\et) \ra
  \Ddgqc(X_\et)$. 
\end{example}

\begin{example}
\label{expl:alg-stack-dg-lifts}
Let $f \colon X \ra Y$ be a concentrated morphism of arbitrary
algebraic stacks over the field $\kk$.
Then the functors
\begin{equation}
\label{eq:derived-push-pull}
\L f^* \colon \Dqc(Y) \rightleftarrows \Dqc(X) \colon \R f_*
\end{equation}
lift to dg~functors 
\begin{equation}
\label{eq:dg-lift-push-pull}
\ulu{f} \colon \Ddgqc(Y) \rightleftarrows \Ddgqc(X) \colon \ull{f}
\end{equation}
between the dg~enhancements of
Example~\ref{expl:algebraic-stack-enhancement} as we explain below.
Given a complex $E \in \Dqc(X)$,
we also get a lift of the triangulated functor 
$(E \otimes^\L -) \colon \Dqc(X) \ra \Dqc(X)$
to a dg functor $(E \ul{\otimes} -) \colon \Ddgqc(X) \ra \Ddgqc(X)$.
Moreover, if the functors $\L f^\ast$, $\R f_\ast$ and $(E \otimes^\L -)$
restrict to the triangulated categories $\Dpf(X)$ and~$\Dpf(Y)$,
then the lifts $\ulu{f}$, $\ull{f}$, $(E \ul{\otimes} -)$ restrict to the
dg~categories $\Ddgpf(X)$ and~$\Ddgpf(Y)$.

Due to the fact that $f$ does not induce a morphism between the
lisse-étale topoi,
this situation is more complicated than the situation
in Example~\ref{expl:DM-stack-dg-lifts}.
To circumvent the problem one can use the technique of 
cohomological descent from \cite{olsson2007} and~\cite{lo2008}.  
Choose smooth hyper-coverings $\pi_X\colon X_\bullet \to X$
and $\pi_Y\colon Y_\bullet \to Y$ 
together with a morphism $X_\bullet \ra Y_\bullet$ over $f \colon
X \ra Y$. Passing to the associated 
strictly simplicial
algebraic spaces
we obtain 
a morphism
$\widetilde{f}\colon X_\bullet^+ \to Y_\bullet^+$ 
augmenting $f$.
This gives us a diagram
$$
\xymatrix{
  (X_\liset, \mathcal{O}_X)&
(X_{\bullet, \liset}^+, \mathcal{O}_X) \ar[l]_{\pi_X} \ar[r]^{\epsilon_X}&
(X_{\bullet, \et}^+, \mathcal{O}_X) \ar[d]^{\widetilde{f}}
\\
(Y_\liset, \mathcal{O}_Y) &
(Y_{\bullet, \liset}^+, \mathcal{O}_Y) \ar[l]_{\pi_Y} \ar[r]^{\epsilon_Y}&
(Y_{\bullet, \et}^+, \mathcal{O}_Y)
\\
}
$$
of ringed topoi, where $\epsilon_X$ and $\epsilon_Y$ are restriction
morphism similar to the morphisms 
\eqref{eq:7}
from Remark~\ref{rem:DM-compare-liset-small-etale}.
The dg~functor $\ulu{f} \colon \Ddg(Y) \to \Ddg(X)$ is defined as the
composition
$$
(\ul\pi_X)_\ast
\ul\epsilon_X^\ast
\ulu{\widetilde{f}}
(\ul \epsilon_Y)_\ast
\ul \pi_Y^\ast
$$
and similarly for $\ull{f}$.
Again using Remark~\ref{rem:dg-lifts-subcategories} we see that
the restrictions of $\ulu{f}$ and $\ull{f}$ to
$\Ddgqc(X)$ and $\Ddgqc(Y)$ give the lifts \eqref{eq:dg-lift-push-pull}
of the triangulated functors \eqref{eq:derived-push-pull}
(cf.\ \cite[Example~2.2.5]{lo2008}, \cite[Section~1]{hr2014}).
\end{example}

\begin{remark}
As we have seen in Example~\ref{expl:alg-stack-dg-lifts},
the triangulated functors $\L f^\ast$, $\R f_\ast$ and $(E \otimes^\L -)$
lift to dg~functors $\ulu{f}$, $\ull{f}$ and $(E \ul\otimes -)$
when working over a field.
Over an arbitrary base ring $R$, dg $R$-linear (or even additive)
replacement
functors similar to 
$\ee$ and $\ii$ need not exist 
(see \cite[Lemma~4.4]{olaf-six-operations-on-dg-enhancements-v1}).
However, it is presumably possible to define morphisms in the 
homotopy category of $R$-linear dg~categories (where the
quasi-equivalences are inverted)
which lift these functors when considered as
morphisms in the homotopy category of triangulated categories
(where equivalences are inverted).
\end{remark}

\subsection{Geometric dg~categories}
\label{sec:nonc-schem}

After recalling some standard notions for dg categories we
discuss geometric dg categories.
Then we state Orlov's gluing result as 
Theorem~\ref{thm:orlov-with-converse}.
We keep the assumption that $\kk$ is a field.

\begin{definition}
  [{cf.\ \cite[Definition~2.4]{toen-vaquie-moduli-objects-dg-cats},
  \cite[Definition~2.3]{toen-finitude-homotopique-propre-lisse},
  \cite[Sections~2.2, 2.5]{valery-olaf-matrix-factorizations-and-motivic-measures}}]
  \label{d:loc-perf-cpt-gen-proper-smooth-(triang-)sat(-fin-type)}
  Let $\mathcal{A}$ be a $\kk$-linear dg category.
  \begin{enumerate}
  \item 
    $\mathcal{A}$ is \define{triangulated} if it is
    pretriangulated and the triangulated category $[\mathcal{A}]$
    is idempotent 
    complete.
  \item
    \label{enum:loc-cohom-bd}
    $\mathcal{A}$ is \define{locally cohomologically bounded}
    if $\mathcal{A}(A,B)$ is cohomologically bounded for all $A$,
    $B \in
    \mathcal{A}$.
  \item
    $\mathcal{A}$ is \define{locally perfect}
    if $\mathcal{A}(A,B)$ is a perfect complex of $\kk$-vector
    spaces, for all $A$, $B \in \mathcal{A}$.
    That is, all complexes $\mathcal{A}(A,B)$ have bounded and
    finite dimensional cohomology.
  \item
    $\mathcal{A}$ \define{has a compact generator} if its
    derived category $\D(\mathcal{A})$ of dg $\mathcal{A}$-modules
    has a compact generator.
  \item
    $\mathcal{A}$ is \define{proper} if it is locally
    perfect and has a compact generator.
  \item
    \label{enum:smooth}
    $\mathcal{A}$ is \define{smooth} if $\mathcal{A}$ is compact
    as an object of the derived category of dg
    $\mathcal{A} \otimes \mathcal{A}^\opp$-modules. 
  \item
    $\mathcal{A}$ is \define{saturated} if it is triangulated,
    smooth and proper.
  \end{enumerate}
\end{definition}

We say that two dg categories are quasi-equivalent if they are
connected by a zig-zag of quasi-equivalences. The above notions
are all well-defined on quasi-equivalence classes of dg
categories (cf.\
\cite[Lemma~2.12]{valery-olaf-matrix-factorizations-and-motivic-measures}). 
In fact, 
properties~\ref{enum:loc-cohom-bd}--\ref{enum:smooth}
are well-defined on Morita equivalence classes of dg categories
(cf.\
\cite[Lemma~2.13]{valery-olaf-matrix-factorizations-and-motivic-measures}).

\begin{remark}
\label{rem:compare-to-orlov}
Orlov's definition of a \define{(derived) non-commutative scheme}
(\cite[Definition~3.3]{orlov2014}) can be reformulated
using the terms above.
A non-commutative scheme is precisely
a locally cohomologically bounded,
triangulated dg~category with a compact generator.

Indeed, a triangulated dg~category $\mathcal{A}$ has a compact generator
if and only if it
is quasi-equivalent to a dg~category of perfect
dg $A$-modules for some dg~algebra $A$
(\cite[Definition~2.2, Lemma~2.3, Corollary~2.4,
Proposition~2.16]{valery-olaf-matrix-factorizations-and-motivic-measures}). In
this case, $\mathcal{A}$ is locally cohomologically bounded
if and only if $A$ is (locally) cohomologically bounded
(the property of being ''locally cohomologically bounded'' can be
added to the list in \cite[Lemma~2.13]{valery-olaf-matrix-factorizations-and-motivic-measures}.
\end{remark}

\begin{definition}
  [{cf.\ \cite[Definition~4.3]{orlov2014}}]
  \label{d:geometric-NC-scheme}
  A dg category $\mathcal{A}$ is \define{geometric}
  if there exists a smooth projective scheme $X$ over $\kk$
  and an admissible subcategory $\mathcal{S}$ of $\Dpf(X)$ such
  that  
  $\mathcal{A}$ and the full dg~subcategory of $\Ddgpf(X)$
  consisting of objects of $\mathcal{S}$ are quasi-equivalent.
\end{definition}

\begin{remark}
\label{rem:geometric-implies-saturated}
Geometric dg~categories are saturated.
Indeed, if $X$ is any
scheme, then
$\Dpf(X)$ is idempotent complete
\cite[\sptag{08GA}]{stacks-project}, and so is any admissible
subcategory. This shows that a geometric dg category is
triangulated.
If $X$ is smooth and proper over the field $\kk$, then
$\Ddgpf(X)$ is smooth and proper,
by \cite[Theorem~1.2 and~1.4]{valery-olaf-new-enhancements} or
\cite[Proposition~3.31]{orlov2014}. Moreover, smoothness and
properness are inherited to dg subcategories of $\Ddgpf(X)$ enhancing admissible
subcategories of $\Dpf(X)$, by
\cite[Proposition~2.20]{valery-olaf-matrix-factorizations-and-motivic-measures}.

This shows, together with 
Remark~\ref{rem:compare-to-orlov},
that our geometric dg categories coincide
with Orlov's geometric noncommutative schemes as defined in
\cite[Definition~4.3]{orlov2014}.
\end{remark}

\begin{example}
Not all geometric dg~categories are of the form $\Ddgpf(X)$ for
some smooth projective variety $X$ over $\kk$.
Let $\Lambda = \kk[\bullet \to \bullet]$ be the path algebra
of a Dynkin quiver of type $A_2$.
Then the standard enhancement of the bounded derived category
$\D^\mathrm{b}(\Lambda)$ is geometric by \cite[Corollary~5.4]{orlov2014}.

On the other hand, the third power of the Serre functor on 
$\D^\mathrm{b}(\Lambda)$ is isomorphic to the shift $[1]$
(cf.~\cite[Proposition~3.1]{hi2011}).
Hence $\D^\mathrm{b}(\Lambda)$ cannot be equivalent to
$\Dpf(X)$ for any smooth projective variety $X$ over $\kk$.
\end{example}

\begin{lemma}
  \label{l:admissible-dg-subcats-of-geometric}
  Let $\mathcal{B}$ be a dg subcategory of a
  geometric 
  dg 
  category $\mathcal{A}$ such that
  $[\mathcal{B}]$ is an admissible
  subcategory of the 
  triangulated category $[\mathcal{A}]$. 
  Then
  $\mathcal{B}$ is geometric.
\end{lemma}

We will see in Corollary~\ref{c:right-admissible-dg-subcats-of-geometric}
below that it is enough to assume that 
$[\mathcal{B}]$ is right or left admissible
in $[\mathcal{A}]$. 

\begin{proof}
  Let $X$ be a smooth projective scheme and  
  $\mathcal{E}$ a dg subcategory of $\Ddgpf(X)$
  such that $[\mathcal{E}]$ is an admissible subcategory of
  $[\Ddgpf(X)]$ and there is a zig-zag of
  quasi-equivalences connecting $\mathcal{A}$ and
  $\mathcal{E}$. 
  Transfering $\mathcal{B}$ along such a zig-zag yields
  a zig-zag of quasi-equivalences connecting $\mathcal{B}$
  with 
  a dg subcategory $\mathcal{F}$ of
  $\mathcal{E}$ such that 
  $[\mathcal{F}]$ is an admissible subcategory of 
  $[\mathcal{E}]$. But then 
  $[\mathcal{F}]$ is also admissible in $[\Ddgpf(X)]$.
  Hence $\mathcal{B}$ is geometric.
\end{proof}

\begin{proposition}
  \label{p:right-admissible-dg-subcats-of-saturated}
  Let $\mathcal{B}$ be a dg subcategory of a
  saturated dg 
  category $\mathcal{A}$ such that
  $[\mathcal{B}]$ is a right (resp.\ left) admissible
  subcategory of the 
  triangulated category $[\mathcal{A}]$. 
  Then $[\mathcal{B}]$
  is admissible in  
  $[\mathcal{A}]$ and 
  $\mathcal{B}$ is saturated.
\end{proposition}

\begin{proof}
  We have a semiorthogonal decomposition 
  $[\mathcal{A}]=\langle [\mathcal{B}]^\perp,
  [\mathcal{B}] \rangle$
  (resp.\ 
  $[\mathcal{A}]=\langle [\mathcal{B}],
  \leftidx{^\perp}{[\mathcal{B}]}{} \rangle$).
  So our claim follows from the proof of
  \cite[Proposition~2.26]{valery-olaf-matrix-factorizations-and-motivic-measures}.
\end{proof}

\begin{corollary}
  \label{c:right-admissible-dg-subcats-of-geometric}
  Let $\mathcal{B}$ be a dg subcategory of a
  geometric 
  dg 
  category $\mathcal{A}$ such that
  $[\mathcal{B}]$ is a right (resp.\ left) admissible
  subcategory of the 
  triangulated category $[\mathcal{A}]$. 
  Then $[\mathcal{B}]$ is admissible in $[\mathcal{A}]$
  and 
  $\mathcal{B}$ is geometric.
\end{corollary}

\begin{proof}
  Since $\mathcal{A}$ is
  saturated, by 
  Remark~\ref{rem:geometric-implies-saturated}, this
  follows from
  Proposition~\ref{p:right-admissible-dg-subcats-of-saturated}
  and 
  Lemma~\ref{l:admissible-dg-subcats-of-geometric}.
\end{proof}

We reformulate Orlov's result 
\cite[Theorem~4.15]{orlov2014} 
that
the gluing of geometric dg categories is again geometric;
for completeness we add the implication in the other direction.

\begin{theorem}
  [{cf.~\cite[Theorem~4.15]{orlov2014}}]
  \label{thm:orlov-with-converse}
  Let $\mathcal{A}$ be a 
  pretriangulated 
  dg~category 
  with full dg subcategories $\mathcal{B}_1, \dots,
  \mathcal{B}_n$ such that 
  $[\mathcal{A}]=\langle [\mathcal{B}_1], \dots,
  [\mathcal{B}_n]\rangle$ is a 
  semiorthogonal decomposition. 
  Then $\mathcal{A}$ is geometric if and only if 
  $\mathcal{A}$
  is locally perfect 
  and 
  all dg
  categories    
  $\mathcal{B}_1, \dots,
  \mathcal{B}_n$ are geometric.

  Moreover, if these conditions are satisfied then 
  all
  $[\mathcal{B}_i]$ are admissible in $[\mathcal{A}]$.
\end{theorem}

\begin{proof}
  If $\mathcal{A}$ is locally perfect and all dg
  categories $\mathcal{B}_1, \dots \mathcal{B}_n$
  are geometric then $\mathcal{A}$ is geometric by
  \cite[Theorem~4.15]{orlov2014}. (The
  ``proper'' dg categories in
  \cite{orlov2014} are usually called
  locally perfect, cf.\
  \cite[Remark~3.15]{orlov2014}).

  Conversely assume that 
  $\mathcal{A}$ is geometric. Then $\mathcal{A}$ is
  saturated by 
  Remark~\ref{rem:geometric-implies-saturated} and in particular
  locally perfect.
  Corollary~\ref{c:right-admissible-dg-subcats-of-geometric}
  and an easy induction (using
  \cite[Lemma~A.11]{valery-olaf-matfak-semi-orth-decomp}) 
  shows that all $\mathcal{B}_i$ are geometric and that all
  $[\mathcal{B}_i]$ are admissible in $[\mathcal{A}]$.
\end{proof}


\section{Geometricity for dg enhancements of algebraic stacks}
\label{sec:geometry}
In this section, we combine Orlov's result on gluing of geometric
dg categories
with a geometric argument to obtain the results about
geometricity for dg enhancements of algebraic stacks stated in the introduction.
The geometric argument depends on the existence of \define{destackifications}
in the sense of \cite{bergh2014, br2015}.
We start by briefly recalling this notion.

In this section, we will mostly work with tame stacks which are separated
and of finite type over a field $\kk$.
Such a stack will be called an 
\define{orbifold} provided that it is smooth over $\kk$ and contains an
open dense substack which is an algebraic space.

Let $X$ be an orbifold over $\kk$.
Although the stack $X$ is smooth,
the same need not hold for its coarse space $X_\coarse$.
However, it is possible to modify the stack via a sequence
of birational modifications such that the coarse space
of the modified stack becomes smooth.
It suffices to use two kinds of modifications:
blowups in smooth centers and root stacks in smooth divisors.
Collectively, we refer to such modifications as \define{smooth stacky blowups}.

\begin{theorem}[Destackification]
\label{thm:destack}
Let $X$ be a tame, separated algebraic stack which is smooth
and of finite type over a field $\kk$.
Assume that $X$ contains an open dense substack which is
an algebraic space. 
Then there exists a morphism $f\colon Y \to X$,
which is a composition of smooth stacky blowups,
such that $Y_\coarse$ is smooth over $\kk$
and such that $Y \to Y_\coarse$ is an iterated root 
construction in an snc divisor $E$ on $Y_\coarse$.
\end{theorem}
\begin{proof}
The case where $X$ has abelian stabilizers is treated
in \cite[Theorem~1.2]{bergh2014}.
In the discussion before \cite[Corollary~1.4]{bergh2014} it is shown how the
abelian hypothesis can be removed if $\kk$ has characteristic zero.
For $\kk$ of arbitrary characteristic,
the theorem is shown in \cite{br2015}.
\end{proof}

\begin{proposition}
  \label{p:pullback-full-faithful-admissible}
  Let $f \colon X \ra Y$ be a concentrated morphism of algebraic
  stacks over a field $\kk$ such that the natural morphism
  $\mathcal{O}_Y \to \R f_\ast \mathcal{O}_X$ is an isomorphism
  and $\R f_*$ preserves perfect complexes. If the $\kk$-linear
  dg category 
  $\Ddgpf(X)$ is  
  geometric then so is $\Ddgpf(Y)$.
\end{proposition}

\begin{proof}
  Lemma~\ref{lem:left-adjoint-ff}
  shows that
  $\L f^* \colon \Dpf(Y) \ra \Dpf(X)$ is full and
  faithful.
  Example~\ref{expl:alg-stack-dg-lifts}
  and Remark~\ref{rem:dg-lifts-subcategories}
  provide the dg functor 
  $\ulu{f} \colon \Ddgpf(Y) \ra \Ddgpf(X)$ 
  lifting $\L f^*$ to dg enhancements.
  It defines a quasi-equivalence from
  $\Ddgpf(Y)$ to $\mathcal{E}$ where
  $(\mathcal{E},\epsilon)$ is the induced dg enhancement of
  the 
  essential image of 
  $\L f^* \colon \Dpf(Y) \ra \Dpf(X)$
  (cf.\ \cite[Lemma
  2.5]{valery-olaf-matrix-factorizations-and-motivic-measures}). 
  By assumption, this
  essential image is a right admissible subcategory.
  Therefore, if $\Ddgpf(X)$ is geometric, so are $\mathcal{E}$ and
  $\Ddgpf(Y)$; 
  either use
  Corollary~\ref{c:right-admissible-dg-subcats-of-geometric}, 
  or the easier Lemma~\ref{l:admissible-dg-subcats-of-geometric}
  together with 
  the fact that
  the essential image of $\L f^*$ is left
  admissible, by Lemma~\ref{lem:left-adjoint-la}.
\end{proof}

\begin{proposition}
\label{p:semi-orthogonal-iterated-geometricity}
Let $X$ be a smooth projective scheme over a field $\kk$.
Assume that 
$E$
is an snc~divisor on $X$ and that
$r > 0$ 
is a multi-index (with respect to the indexing set of $E$).
Then the $\kk$-linear dg~category $\Ddgpf(X_{r^{-1}E})$
associated to the root stack $X_{r^{-1}E}$ is geometric. 
\end{proposition}
\begin{proof}
The root stack $X_{r^{-1}E}$ is tame
(Proposition~\ref{prop:basic-root}.\ref{enum:tame})
and proper
(Proposition~\ref{prop:basic-root}.\ref{enum:pi-proper-fp-birat})
over $\kk$.
This implies that the cohomology of coherent sheaves is
bounded (\cite[Theorem~2.1]{hr2015}) and coherent
(\cite[Theorem~1]{falting-finiteness-coherent}).
Therefore, the dg~category $\Ddgpf(X_{r^{-1}E})$ is locally
perfect.

Recall 
the
semiorthogonal decomposition for iterated root
constructions from Theorem~\ref{thm:semi-orthogonal-iterated}
and observe that the functors
\eqref{eq:transform-functor}
involved in this decomposition
lift to dg functors
\begin{equation}
\label{eq:6}
\mathcal{O}_{X_{r^{-1}E}}(ar^{-1}E)
\ul{\otimes}
(\ul{\iota_a})_\ast(\ul{\rho_a})^\ast(-)
\colon
\Ddgpf(E(I_a)) \to \Ddgpf(X_{r^{-1}E})
\end{equation}
between dg enhancements,
by Example~\ref{expl:alg-stack-dg-lifts}, 
Lemma~\ref{lem:right-adjoint} and
Remark~\ref{rem:dg-lifts-subcategories}.
Since all intersections $E(I_a)$ are smooth, projective schemes
over $\kk$,
all dg categories $\Ddgpf(E(I_a))$ are geometric.
The claim now follows from Orlov's glueing Theorem~\ref{thm:orlov-with-converse}.
\end{proof}

We are now ready to prove our first result on geometricity for
dg enhancements of algebraic stacks.

\begin{theorem}
\label{thm:main-tame}
Let $X$ be a tame, smooth, projective algebraic stack
over an arbitrary field $\kk$.
Then the $\kk$-linear dg category $\Ddgpf(X)$ is geometric, and in particular saturated.
\end{theorem}
\begin{proof}
Since $X$ is a global quotient stack,
there is a projectivized vector bundle $P \to X$
such that $P$ contains an open dense substack which is an algebraic space
(cf.~\cite[Proof of Theorem~1]{kv2004}).
Explicitly, we can construct such a bundle as follows.
Let $T \to X$ be a $\GL_n$-torsor where $T$ is an algebraic
space. 
Consider the corresponding vector
bundle 
$\mathcal{E}$ 
of rank $n$ on $X$. 
Then we have dense open immersions $T \hookrightarrow V \hookrightarrow P$,
where $V = \VV(\sheafEnd_{\mathcal{O}_X}(\mathcal{E}))$
and $P = \PP(\sheafEnd_{\mathcal{O}_X}(\mathcal{E})\oplus\mathcal{O}_X)$.
The stack $P$ is tame since $P \to X$ is representable.
Since also $T$ is representable,
it follows that $P$ is an orbifold.

Now we apply Theorem~\ref{thm:destack} and get a proper
birational morphism 
$Y \to P$ which is a composition of smooth stacky blowups such that
$Y$ and $Y_\coarse$ are smooth and 
the canonical map $Y \to Y_\coarse$ is an iterated root
construction in an 
snc divisor $E$ on $Y_\coarse$.

Note that $Y$ is a projective algebraic stack.
Indeed, the map $Y \to P$ is a composition of root stacks and
blowups and $P\to X$ is projective,
so this follows from Lemma~\ref{lem:projective-root}
and~Lemma~\ref{lem:projective-stack}.
In particular, $Y_\coarse$ is a smooth projective scheme.
Hence
Proposition~\ref{p:semi-orthogonal-iterated-geometricity}
shows that $\Ddgpf(Y)$ is geometric.

Denote the composition $Y \to P \to X$ by $\pi$.
Since $\pi$ is a composition of root stacks,
blow-ups and the structure morphism of a projective bundle,
the canonical morphism $\mathcal{O}_X \to \R\pi_\ast\mathcal{O}_Y$
is an isomorphism,
and $\R\pi_\ast$ preserves perfect complexes by
part
\ref{enum:blowups}, 
\ref{enum:proj-bundle},
\ref{it:fully-faithful-root}
of
Example~\ref{ex:fully-faithful}.
In particular, the morphism $\pi$ satisfies the
assumptions of Proposition~\ref{p:pullback-full-faithful-admissible}.
Therefore $\Ddgpf(X)$ is geometric since the same holds for $\Ddgpf(Y)$. 
\end{proof}

If we work over a field $\kk$ which admits resolution of singularities,
we have the following version of Chow's Lemma.
\begin{proposition}[Chow's Lemma]
\label{pro:chow}
Let $X$ be a separated Deligne--Mumford stack which is smooth and of finite type
over a field $\kk$ of characteristic zero.
Then there exists a morphism $\pi\colon Y \to X$ which
is a composition of (non-stacky) blowups in smooth centers such
that $Y$ is a quasi-projective algebraic stack.
\end{proposition}
\begin{proof}
By \cite[Theorem~4.3]{choudhury2012},
which is attributed to Rydh,
we can find a sequence of (non-stacky) blowups $Y \to \cdots \to X$
in smooth centers such that $Y_\coarse$ is quasi-projective.
By \cite[4.4]{kresch2009} any smooth Deligne--Mumford stack of finite
type over a field is automatically a global quotient if its coarse
space is quasi-projective.
In particular, the stack $Y$ is quasi-projective.  
\end{proof}

In particular, over a field of characteristic zero,
we can replace the projectivity assumption 
from
Theorem~\ref{thm:main-tame}
by a properness
assumption.

\begin{theorem}
\label{thm:main-zero}
Let $X$ be a smooth, proper Deligne--Mumford stack over a
field $\kk$ of characteristc zero.
Then the $\kk$-linear dg category $\Ddgpf(X)$ is geometric, and in particular saturated.
\end{theorem}
\begin{proof}
By Proposition~\ref{pro:chow},
there is a composition $\pi\colon Y \to X$
of blow-ups in smooth centers such that $Y$ is a projective
algebraic stack.
Hence $\Ddgpf(Y)$ is geometric by Theorem~\ref{thm:main-tame}.
By arguing as in the final paragraph of the proof of the same theorem,
we see that also $\Ddgpf(X)$ is geometric.
\end{proof}


\appendix
\section{Bounded derived category of coherent modules}
\label{sec:bound-deriv-coherent}

Our aim is to show that the bounded derived category of coherent
modules on a regular, quasi-compact, separated algebraic stack
with finite stabilizers is equivalent to the derived category of
perfect complexes (see 
Remark~\ref{rem:D-Coh-and-DbCoh=Dpf}).

The category of coherent
$\mathcal{O}_X$-modules on a locally noetherian algebraic stack
$X$ is denoted by $\Coh(X)$.
We use the usual decorations for full subcategories of derived
categories. For example, the symbol $\D^-_\Coh(\Qcoh(X))$ denotes
the full subcategory of the derived category $\D(\Qcoh(X))$ of 
quasi-coherent modules whose objects have bounded above coherent
cohomology modules. 

The following proposition generalizes a well-known result for 
noetherian schemes \cite [Exposé~II,
Proposition~2.2.2]{berthelot-grothendieck-illusie-SGA-6},
\cite[Proposition~3.5]{huybrechts2006} to noetherian algebraic
stacks.
  
\begin{proposition}
  \label{p:D-coh-noetherian}
  Let $X$ be an noetherian algebraic stack. Then the obvious
  functor defines an
  equivalence
  \begin{equation*}
    \D^-(\Coh(X)) \sira \D^-_\Coh(\Qcoh(X)).
  \end{equation*}
\end{proposition}

\begin{proof}
  It is certainly enough to show that each bounded above complex
  of quasi-coherent modules with coherent cohomology modules
  has a quasi-isomorphic 
  subcomplex of coherent modules.
  This is an easy consequence of the proof of
  \cite[Proposition~3.5]{huybrechts2006} as soon as we know the
  following fact: given any
  epimorphism $G \ra F$ from a quasi-coherent
  module $G$ to a coherent module $F$, there is a coherent
  submodule $G'$ of $G$ such that the composition $G' \subset G
  \ra F$ is still an epimorphism. This latter statement follows
  from the fact that every quasi-coherent module is
  the filtered colimit of its coherent 
  submodules~\cite[Proposition~15.4]{lmb2000}
  and \cite [Exposé~II,
  Lemma~2.1.1.a)]{berthelot-grothendieck-illusie-SGA-6}.
\end{proof}

\begin{proposition}
  \label{p:regular-DbCoh=Dpf}
  Let $X$ be a regular and quasi-compact algebraic stack.
  Then we have an equality $\Dpf(X) = \D^{\mathrm{b}}_\Coh(X)$.
\end{proposition}

\begin{proof}
  Since $X$ is quasi-compact we have 
  $\Dpf(X) \subset \D^{\mathrm{b}}_\Coh(X)$.
  In order to show equality it is enough to prove that any coherent
  module is perfect. Let $\Spec A \ra X$ be any smooth morphism
  where $A$ is a ring. Then $A$ is regular.
  It is enough to prove that any finitely generated
  $A$-module $M$ has a finite resolution by finitely generated
  projective 
  $A$-modules. Let $P \ra M$ be a resolution by finitely
  generated projective $A$-modules. 
  Let $\mfp \in \Spec A$.
  Since $A_\mfp$ is regular,
  it has finite global dimension by the
  Auslander--Buchsbaum--Serre theorem.
  Therefore, there is a natural number $n=n(\mfp)$ such that the
  kernel 
  of the differential $d^{-n} \colon (P^{-n})_\mfp \ra (P^{-n+1})_\mfp$ is
  a finitely generated projective $A_\mfp$-module.
  Since $A$ is noetherian, there is some open neighborhood $\Spec
  A_f$ of
  $\mfp$ in $\Spec A$ such that the kernel of  
  $d^{-n} \colon (P^{-n})_f \ra (P^{-n+1})_f$ is a
  finitely generated projective $A_f$-module.
  Then also all kernels 
  $d^{-i} \colon (P^{-i})_f \ra (P^{-i+1})_f$,
  for $i \geq n$, 
  are
  finitely generated projective $A_f$-modules.
  Since $\Spec A$ is quasi-compact there is a natural number $N$
  such that the kernel of $d^{-N} \colon P^{-N} \ra P^{-N+1}$ is
  a finitely generated projective $A$-module.
\end{proof}

\begin{remark}
  \label{rem:D-Coh-and-DbCoh=Dpf}
  If $X$ is a noetherian, separated algebraic stack
  with finite 
  stabilizers we have equivalences
  \begin{equation*}
    \D^-(\Coh(X)) \sira \D^-_\Coh(\Qcoh(X)) \sira \D^-_\Coh(X).
  \end{equation*}
  This follows immediately from
  Proposition~\ref{p:D-coh-noetherian} and the equivalence
  $\D(\Qcoh(X)) \sira \Dqc(X)$ from \eqref{eq:10}.
  If we assume in addition that $X$ is regular then  
  Proposition~\ref{p:regular-DbCoh=Dpf} together with the above
  equivalences shows that 
  \begin{equation*}
    \D^{\mathrm{b}}(\Coh(X)) \sira \D^{\mathrm{b}}_\Coh(X) = \Dpf(X)
  \end{equation*}
  is an equivalence.
\end{remark}




\bibliographystyle{myalpha}
\bibliography{references}

\def\cprime{$'$} \def\cprime{$'$} \def\cprime{$'$} \def\cprime{$'$}
  \def\Dbar{\leavevmode\lower.6ex\hbox to 0pt{\hskip-.23ex \accent"16\hss}D}
  \def\cprime{$'$} \def\cprime{$'$}
\begin{thebibliography}{BVdB03}

\bibitem[AGV08]{agv2008}
Dan Abramovich, Tom Graber, and Angelo Vistoli.
\newblock Gromov-{W}itten theory of {D}eligne-{M}umford stacks.
\newblock {\em Amer. J. Math.}, 130(5):1337--1398, 2008.

\bibitem[AOV08]{aov2008}
Dan Abramovich, Martin Olsson, and Angelo Vistoli.
\newblock Tame stacks in positive characteristic.
\newblock {\em Ann. Inst. Fourier (Grenoble)}, 58(4):1057--1091, 2008.

\bibitem[BC10]{bc2010}
Arend Bayer and Charles Cadman.
\newblock Quantum cohomology of {$[\Bbb C^N/\mu_r]$}.
\newblock {\em Compos. Math.}, 146(5):1291--1322, 2010.

\bibitem[Ber14]{bergh2014}
Daniel Bergh.
\newblock Functorial destackification of tame stacks with abelian stabilisers.
\newblock \href{http://arxiv.org/abs/1409.5713}{arXiv:1409.5713v1}, 2014.

\bibitem[BK89]{bondal-kapranov-representable-functors}
A.~I. Bondal and M.~M. Kapranov.
\newblock Representable functors, {S}erre functors, and
  reconstructions/mutations.
\newblock {\em Izv. Akad. Nauk SSSR Ser. Mat.}, 53(6):1183--1205, 1337, 1989.

\bibitem[BR15]{br2015}
Daniel Bergh and David Rydh.
\newblock Functorial destackification and weak factorization of orbifolds.
\newblock In preparation, 2015.

\bibitem[BS16]{bs2016}
Daniel Bergh and Olaf Schnürer.
\newblock Conservative descent for semiorthogonal decompositions.
\newblock In preparation, 2016.

\bibitem[BVdB03]{bv2003}
A.~Bondal and M.~Van~den Bergh.
\newblock Generators and representability of functors in commutative and
  noncommutative geometry.
\newblock {\em Mosc. Math. J.}, 3(1):1--36, 258, 2003.

\bibitem[Cad07]{cadman2007}
Charles Cadman.
\newblock Using stacks to impose tangency conditions on curves.
\newblock {\em Amer. J. Math.}, 129(2):405--427, 2007.

\bibitem[Cho12]{choudhury2012}
Utsav Choudhury.
\newblock Motives of {D}eligne-{M}umford stacks.
\newblock {\em Adv. Math.}, 231(6):3094--3117, 2012.

\bibitem[CS15]{canonaco-stellari-uniqueness-of-dg-enhancements}
Alberto Canonaco and Paolo Stellari.
\newblock Uniqueness of dg enhancements for the derived category of a
  {G}rothendieck category, 2015.
\newblock \href{http://arxiv.org/abs/1507.05509}{arxiv:1507.05509v2}.

\bibitem[CT12]{ct2012}
Denis-Charles Cisinski and Gonçalo Tabuada.
\newblock Symmetric monoidal structure on non-commutative motives.
\newblock {\em J. K-Theory}, 9(2):201--268, 2012.

\bibitem[EGAII]{EGAII}
A.~Grothendieck.
\newblock \'{E}l\'ements de g\'eom\'etrie alg\'ebrique. {II}. \'{E}tude globale
  \'el\'ementaire de quelques classes de morphismes.
\newblock {\em Inst. Hautes \'Etudes Sci. Publ. Math.}, (8):222, 1961.

\bibitem[Fal03]{falting-finiteness-coherent}
Gerd Faltings.
\newblock Finiteness of coherent cohomology for proper fppf stacks.
\newblock {\em J. Algebraic Geom.}, 12(2):357--366, 2003.

\bibitem[FMN10]{fmn2010}
Barbara Fantechi, Etienne Mann, and Fabio Nironi.
\newblock Smooth toric {D}eligne-{M}umford stacks.
\newblock {\em J. Reine Angew. Math.}, 648:201--244, 2010.

\bibitem[Hal14]{hall2014}
Jack Hall.
\newblock The {B}almer spectrum of a tame stack.
\newblock 2014.
\newblock \href{http://arxiv.org/abs/1411.6295}{arXiv:1411.6295v1}, to appear
  in Annals of K-theory.

\bibitem[HI11]{hi2011}
Martin Herschend and Osamu Iyama.
\newblock {$n$}-representation-finite algebras and twisted fractionally
  {C}alabi-{Y}au algebras.
\newblock {\em Bull. Lond. Math. Soc.}, 43(3):449--466, 2011.

\bibitem[HLP15]{hlp2015}
Daniel Halpern-Leistner and Daniel Pomerleano.
\newblock Equivariant hodge theory and noncommutative geometry.
\newblock {\em Preprint}, 2015.
\newblock \href{http://arxiv.org/abs/1507.01924}{arXiv:1507.01924v1}.

\bibitem[HNR14]{hnr2014}
Jack Hall, Amnon Neeman, and David Rydh.
\newblock One positive and two negative results for derived categories of
  algebraic stacks, 2014.
\newblock \href{http://arxiv.org/abs/1405.1888}{arXiv:1405.1888v2}.

\bibitem[HR14]{hr2014}
Jack Hall and David Rydh.
\newblock Perfect complexes on algebraic stacks, 2014.
\newblock \href{http://arxiv.org/abs/1405.1887}{arXiv:1405.1887v2}.

\bibitem[HR15]{hr2015}
Jack Hall and David Rydh.
\newblock Algebraic groups and compact generation of their derived categories
  of representations.
\newblock {\em Indiana Univ. Math. J.}, 64:1903--1923, 2015.

\bibitem[Huy06]{huybrechts2006}
D.~Huybrechts.
\newblock {\em Fourier-{M}ukai transforms in algebraic geometry}.
\newblock Oxford Mathematical Monographs. The Clarendon Press Oxford University
  Press, Oxford, 2006.

\bibitem[IU11]{iu2011}
Akira Ishii and Kazushi Ueda.
\newblock The special {M}c{K}ay correspondence and exceptional collection,
  2011.

\bibitem[Kel06]{keller2006}
Bernhard Keller.
\newblock On differential graded categories.
\newblock In {\em International {C}ongress of {M}athematicians. {V}ol. {II}},
  pages 151--190. Eur. Math. Soc., Z\"urich, 2006.

\bibitem[KM97]{km1997}
Se{\'a}n Keel and Shigefumi Mori.
\newblock Quotients by groupoids.
\newblock {\em Ann. of Math. (2)}, 145(1):193--213, 1997.

\bibitem[Kol97]{kollar1997}
J{\'a}nos Koll{\'a}r.
\newblock Quotient spaces modulo algebraic groups.
\newblock {\em Ann. of Math. (2)}, 145(1):33--79, 1997.

\bibitem[Kre09]{kresch2009}
Andrew Kresch.
\newblock On the geometry of {D}eligne-{M}umford stacks.
\newblock In {\em Algebraic geometry---{S}eattle 2005. {P}art 1}, volume~80 of
  {\em Proc. Sympos. Pure Math.}, pages 259--271. Amer. Math. Soc., Providence,
  RI, 2009.

\bibitem[KV04]{kv2004}
Andrew Kresch and Angelo Vistoli.
\newblock On coverings of {D}eligne-{M}umford stacks and surjectivity of the
  {B}rauer map.
\newblock {\em Bull. London Math. Soc.}, 36(2):188--192, 2004.

\bibitem[LMB00]{lmb2000}
G{\'e}rard Laumon and Laurent Moret-Bailly.
\newblock {\em Champs alg\'ebriques}, volume~39 of {\em Ergebnisse der
  Mathematik und ihrer Grenzgebiete. 3. Folge.}
\newblock Springer-Verlag, Berlin, 2000.

\bibitem[LO08]{lo2008}
Yves Laszlo and Martin Olsson.
\newblock The six operations for sheaves on {A}rtin stacks. {I}. {F}inite
  coefficients.
\newblock {\em Publ. Math. Inst. Hautes \'Etudes Sci.}, (107):109--168, 2008.

\bibitem[LO10]{lunts-orlov-enhancement}
Valery~A. Lunts and Dmitri~O. Orlov.
\newblock Uniqueness of enhancement for triangulated categories.
\newblock {\em J. Amer. Math. Soc.}, 23(3):853--908, 2010.

\bibitem[LS12]{valery-olaf-matfak-semi-orth-decomp}
Valery~A. Lunts and Olaf~M. Schn\"urer.
\newblock Matrix-factorizations and semi-orthogonal decompositions for
  blowing-ups.
\newblock {\em accepted by J. Noncommut. Geom.}, 2012.
\newblock \href{http://arxiv.org/abs/1212.2670}{arXiv:1212.2670v2}.

\bibitem[LS13]{valery-olaf-matrix-factorizations-and-motivic-measures}
Valery~A. Lunts and Olaf~M. Schn\"urer.
\newblock Matrix factorizations and motivic measures.
\newblock {\em accepted by J. Noncommut. Geom.}, 2013.
\newblock \href{http://arxiv.org/abs/1310.7640}{arXiv:1310.7640v2}.

\bibitem[LS14]{valery-olaf-new-enhancements}
Valery~A. Lunts and Olaf~M. Schn\"urer.
\newblock New enhancements of derived categories of coherent sheaves and
  applications.
\newblock {\em accepted by J. Algebra}, 2014.
\newblock \href{http://arxiv.org/abs/1406.7559}{arxiv:1406.7559v2}.

\bibitem[MFK94]{mfk1994}
D.~Mumford, J.~Fogarty, and F.~Kirwan.
\newblock {\em Geometric invariant theory}, volume~34 of {\em Ergebnisse der
  Mathematik und ihrer Grenzgebiete (2) [Results in Mathematics and Related
  Areas (2)]}.
\newblock Springer-Verlag, Berlin, third edition, 1994.

\bibitem[Ols07]{olsson2007}
Martin Olsson.
\newblock Sheaves on {A}rtin stacks.
\newblock {\em J. Reine Angew. Math.}, 603:55--112, 2007.

\bibitem[Ols12]{olsson2012}
Martin Olsson.
\newblock Integral models for moduli spaces of {$G$}-torsors.
\newblock {\em Ann. Inst. Fourier (Grenoble)}, 62(4):1483--1549, 2012.

\bibitem[Orl92]{orlov1992}
D.~O. Orlov.
\newblock Projective bundles, monoidal transformations, and derived categories
  of coherent sheaves.
\newblock {\em Izv. Ross. Akad. Nauk Ser. Mat.}, 56(4):852--862, 1992.

\bibitem[Orl14]{orlov2014}
Dmitri Orlov.
\newblock Smooth and proper noncommutative schemes and gluing of dg categories.
\newblock {\em Preprint}, 2014.
\newblock \href{http://arxiv.org/abs/1402.7364}{arXiv:1402.7364v5}.

\bibitem[Ric10]{riche2010}
Simon Riche.
\newblock Koszul duality and modular representations of semisimple {L}ie
  algebras.
\newblock {\em Duke Math. J.}, 154(1):31--134, 2010.

\bibitem[Ryd13]{rydh2013}
David Rydh.
\newblock Existence and properties of geometric quotients.
\newblock {\em J. Algebraic Geom.}, 22(4):629--669, 2013.

\bibitem[Ryd15a]{rydh2015}
David Rydh.
\newblock Approximation of sheaves on algebraic stacks.
\newblock {\em J. Algebraic Geom.}, page~21, 2015.

\bibitem[Ryd15b]{mo206117}
David Rydh.
\newblock Do line bundles descend to coarse moduli spaces of {A}rtin stacks
  with finite inertia?
\newblock MathOverflow, 2015.
\newblock
  \href{http://mathoverflow.net/q/206117}{http://mathoverflow.net/q/206117}
  (version: 2015-05-09).

\bibitem[Sch15]{olaf-six-operations-on-dg-enhancements-v1}
Olaf~M. Schn\"urer.
\newblock Six operations on dg enhancements of derived categories of sheaves.
\newblock {\em Preprint}, 2015.
\newblock \href{http://arxiv.org/abs/1507.08697}{arXiv:1507.08697v1}.

\bibitem[SGA1]{SGA-1}
Alexander Grothendieck.
\newblock {\em Rev\^etements \'etales et groupe fondamental ({SGA} 1)}.
\newblock Documents Math\'ematiques (Paris) [Mathematical Documents (Paris)],
  3. Soci\'et\'e Math\'ematique de France, Paris, 2003.
\newblock S{\'e}minaire de g{\'e}om{\'e}trie alg{\'e}brique du Bois Marie
  1960--61. Updated and annotated reprint of the 1971 original [Lecture Notes
  in Math., 224, Springer, Berlin].

\bibitem[SGA6]{berthelot-grothendieck-illusie-SGA-6}
P.~Berthelot, A.~Grothendieck, and L.~Illusie.
\newblock {\em Th\'eorie des intersections et th\'eor\`eme de
  {R}iemann-{R}och}.
\newblock Lecture Notes in Mathematics, Vol. 225. Springer-Verlag, Berlin,
  1971.
\newblock S{\'e}minaire de G{\'e}om{\'e}trie Alg{\'e}brique du Bois-Marie
  1966--1967 (SGA 6).

\bibitem[SP16]{stacks-project}
The {Stacks Project Authors}.
\newblock Stacks project.
\newblock \url{http://stacks.math.columbia.edu}, 2016.

\bibitem[Tab11]{tabuada2011}
Gonçalo Tabuada.
\newblock A guided tour through the garden of noncommutative motives.
\newblock \href{http://arxiv.org/abs/1108.3787}{arXiv:1108.3787v1}, 2011.

\bibitem[To{\"e}09]{toen-finitude-homotopique-propre-lisse}
Bertrand To{\"e}n.
\newblock Finitude homotopique des dg-alg\`ebres propres et lisses.
\newblock {\em Proc. Lond. Math. Soc. (3)}, 98(1):217--240, 2009.

\bibitem[To{\"e}11]{toen2008}
Bertrand To{\"e}n.
\newblock Lectures on {DG}-categories.
\newblock In {\em Topics in algebraic and topological {$K$}-theory}, volume
  2008 of {\em Lecture Notes in Math.}, pages 243--302. Springer, Berlin, 2011.

\bibitem[TV07]{toen-vaquie-moduli-objects-dg-cats}
Bertrand To{\"e}n and Michel Vaqui{\'e}.
\newblock Moduli of objects in dg-categories.
\newblock {\em Ann. Sci. \'Ecole Norm. Sup. (4)}, 40(3):387--444, 2007.

\end{thebibliography}

\end{document}